\documentclass[journal]{IEEEtran}

\usepackage[ruled,linesnumbered]{algorithm2e}
\usepackage{subfigure}
\usepackage{epsfig} 
\usepackage{mathptmx} 
\usepackage{times} 
\usepackage{amsmath} 
\usepackage{amssymb}  
\usepackage{amsthm}
\usepackage{makeidx}
\usepackage{float}
\usepackage{balance}
\usepackage{booktabs}
\usepackage{bbm}
\usepackage{bm}
\usepackage{caption}
\usepackage{nomencl}
\usepackage{dblfloatfix}
\usepackage{mathtools}

\newcommand{\md}[1]{[#1]_\text{m}}
\newcommand{\vc}[1]{[#1]^\mathsf{v}}

\newcommand{\dD}{\mathsf D}
\newcommand{\dexp}{\text{dexp}}
\newcommand{\ad}{\text{ad}}
\newcommand{\Ad}{\text{Ad}}
\newcommand{\vd}[1]{[#1]_\mathsf{v}}

\newtheorem{theorem}{Theorem}[section]
\newtheorem{lemma}{Lemma}[section]
\newtheorem{remark}{Remark}[section]

\newtheorem{ass}{Assumption}[section]
\newtheorem{cor}{Corollary}[section]
\IEEEoverridecommandlockouts                              
\usepackage{etoolbox}
\makeatletter
\patchcmd{\@makecaption}
  {\scshape}
  {}
  {}
  {}
\makeatother
\usepackage[]{algpseudocode}

\begin{document}
%
\title{Differential Dynamic Programming on Lie Groups: Derivation, Convergence Analysis and Numerical Results}
%
%
%
%

\author{ George I. Boutselis and Evangelos Theodorou
\thanks{The authors are with the school of Aerospace Engineering, Georgia Institute of Technology, Atlanta, GA, USA.
     }
}

\markboth{Journal of \LaTeX\ Class Files,~Vol.~13, No.~9, September~2014}%
{Shell \MakeLowercase{\textit{et al.}}: Bare Advanced Demo of IEEEtran.cls for Journals}

\IEEEtitleabstractindextext{%
\begin{abstract}

 We develop a discrete-time optimal control framework for systems evolving on Lie groups. Our work generalizes the original Differential Dynamic Programming method, by employing a coordinate-free, Lie-theoretic approach for its derivation. A key element lies, specifically, in the use of quadratic expansion schemes for cost functions and dynamics defined on manifolds. The obtained algorithm iteratively optimizes local approximations of the control problem, until reaching a (sub)optimal solution. On the theoretical side, we also study the conditions under which convergence is attained. Details about the behavior and implementation of our method are provided through a simulated example on $TSO(3)$.

\end{abstract}

}

\maketitle


\IEEEdisplaynontitleabstractindextext

%
\IEEEpeerreviewmaketitle

\section{Introduction}
Real physical systems often admit complex configuration spaces. It can be shown, for example, that the set of rigid body transformations behaves as a differentiable manifold \cite{chirbook}. Compared to vector spaces, non-flat manifolds require a more elegant treatment. The field of geometric mechanics provides a framework for studying the motion of such systems, by employing concepts from differential geometry and Lie group theory \cite{bookgeomech}. One of its branches, discrete mechanics, addresses related phenomena in the discrete time/space domain, including geometric integration techniques and variational methods \cite{haier, discrete, hpintegrators}.

Over the past decades, there has been substantial effort to extend modern control theory to systems evolving on smooth manifolds. A plethora of theoretical results spanning stability analysis, controllability, and feedback control design can be found, for example, in \cite{bookgeometriccontrol}. Geometric optimal control, in particular, studies the formulation and solution of optimal control problems for Lie-theoretic representations of dynamic systems. This involves interpreting standard frameworks, such as Dynamic Programming and Pontryagin's maximum principle, using differential-geometric ideas \cite{bookgeomoptimalcontrol}.

Recently, several papers have been published on the development of geometric control algorithms for mechanical systems. In \cite{kobpontr} a numerical method was proposed based on Pontryagin's principle. \cite{kobilarovdiscrete} performed discrete optimal control on a variational integration scheme, by using an off-the-shelf direct optimization solver. In \cite{kobilarovddplie} some standard trajectory-optimization schemes (e.g., the stage-wise Newton method) were modified to account for matrix Lie group representations. Moreover, the work in \cite{lieprojection} derived the projection operator framework on Lie groups, and applied it for continuous-time trajectory-optimization problems. Therein, the importance of covariant differentiation in controls and dynamics applications was highlighted. Finally, the works in \cite{blochso3, leethesis} utilized the necessary optimality conditions to control certain classes of mechanical systems (e.g., the planar pendulum on $SO(2)$, or the single rigid body on $SE(3)$). The major drawback of the latter approaches lies in their heavy dependence on problem specifications.

Differential Dynamic Programming (DDP) was proposed by Mayne and Jacobson for solving discrete and continuous optimal control problems \cite{ddp}. Since then, it has found applications in many complex, high-dimensional, engineering problems (see, for example, \cite{recd, conlddp, Lantoine2012, Lin1991_Part1, Lin1991_Part2}). Its scalability, fast convergence rate, and feedback control policies constitute some of its major attributes. DDP was first mentioned in the context of geometric control in \cite{kobilarovddplie}, and was later used in \cite{liegroupddp_estimation} for estimation. Unfortunately, these works briefly provided the final form of the algorithm given certain matrix Lie groups and problem definitions. Moreover, the corresponding expressions were justified in the basis of using the traditional, vector-based formulation from \cite{ddp}. To the best of the authors' knowledge, the controls literature is lacking a rigorous work on the extension of DDP to non-flat configuration manifolds.

In this paper we formulate a Lie-theoretic version of DDP, and cover topics spanning the development, convergence analysis, and numerical behavior of the algorithm. Specifically, our contributions can be listed as follows:
\begin{itemize}
\item The derivation goes along the lines of the original DDP method, with each step being modified to account for Lie group formulations of dynamics and cost functions. We thus obtain a numerical, iterative, coordinate-free algorithm that generalizes the works in \cite{ddp}, \cite{kobilarovddplie} and \cite{liegroupddp_estimation}.
\item We provide linearization schemes for generic classes of discrete mechanical systems. In contrast to \cite{kobilarovddplie, liegroupddp_estimation}, we consider second-order expansions that help us increase the convergence rate of DDP.
\item An extensive analysis is included that studies the convergence properties of our framework. As in the Euclidean case, we find that the algorithm will converge to a (sub)optimal solution, when a Hessian-like operator remains positive definite over the entire optimization sequence.
\item We discuss practical issues through numerical simulations on $TSO(3)$. Details are provided for transitioning between the Lie-theoretic formulation of the algorithm, and its corresponding matrix/vector representation. Additionally, we highlight DDP's benefits over standard trajectory-optimization schemes.
\end{itemize}

The remaining of this paper is organized as follows: Section \ref{sec:prel} introduces the notation used in this work, and gives some preliminaries on Lie group theory and differential geometry. Section \ref{sec:ddp_der} defines the optimal control problem, and derives the Differential Dynamic Programming algorithm on Lie groups. In section \ref{sec:zetaexp} we develop linearization schemes for discrete mechanical systems, which we also express in matrix/vector form for implementation purposes. Section \ref{sec:conv} provides some convergence-related results for our algorithm. In section \ref{sec:sim} we validate the applicability of our methodology by controlling a mechanical system in simulation. Finally, section \ref{sec:conc} is the conclusion.

\section{Preliminaries and Notation}\label{sec:prel}
Here, we explain the notation used in our paper, and review certain concepts from differential geometry and Lie group theory. These can be found in any standard textbook, such as \cite{bookgeometriccontrol}, \cite{Gallier_notes} and \cite{Absbook}.

We denote by $G$ the Lie group that corresponds to the configuration space of a dynamical system. We let $e$ be its identity element, and define $L_h:G\rightarrow G$ (respectively, $R_h:G\rightarrow G$) as the left (respectively, right) translation map, for all $h\in G$. The tangent and cotangent bundles of $G$ are denoted by $TG$ and $T^*G$, respectively, while $\mathfrak{g}:=T_eG$ and $\mathfrak{g}^*:=T^*_eG$ correspond to the Lie algebra and its dual. The tangent map of $L_h$ (resp., $R_h$) at $g\in G$ is written as $T_gL_h:TG\rightarrow TG$ (resp., $T_gR_h:TG\rightarrow TG$). We shall occasionally write for brevity $hg$ and $g\xi$, $\xi g$, instead of $L_hg$ and $T_eL_g\xi$, $T_eR_g\xi$, for all $g,h\in G$, $\xi\in\mathfrak{g}$. Lastly, let $\mathfrak{X}$ represent the set of smooth vector fields on G. Then, for any smooth function $\mathsf f:G\rightarrow\mathbb{R}$, we define the {\it Lie bracket} $[\cdot,\cdot]:\mathfrak{X}\times \mathfrak{X}\rightarrow \mathfrak{X}$, such that $[X,Y](\mathsf f):=X(Y(\mathsf f))-Y(X(\mathsf f))$.

To proceed, we will also make use of the following notions:

\textbf{Natural pairing and dual maps.} Given a vector space $V$ and its dual $V^*$, we define their {\it natural pairing} as the bilinear map $\langle\cdot,\cdot\rangle:V^*\times V\rightarrow\mathbb{R}$, such that $\langle\phi,x\rangle:=\phi(x)$ for each $x\in V$, $\phi\in V^*$. Moreover, for any linear map $\mathsf f:V\rightarrow W$ between vector spaces, we define its {\it dual}, $\mathsf f^*:W^*\rightarrow V^*$, by imposing the property: $\langle\psi,\mathsf f(x)\rangle=\langle\mathsf f^*\circ\psi,x\rangle$, for each $\psi\in W^*$. Note that the former pairing is defined on $(W^*,W)$, while the latter on $(V^*,V)$. When $W=V^*$ and $\mathsf f=\mathsf f^*$, we say that $\mathsf{f}$ is a {\it symmetric} map. This implies the canonical identification of $V$ with its bidual, $V^{**}$.

\textbf{Affine connections.} Let $X, Y\in\mathfrak{X}$ be two vector fields on a Lie group $G$. Given an affine connection, $\nabla:\mathfrak{X}\times \mathfrak{X}\rightarrow \mathfrak{X}$, we denote the covariant derivative of $Y$ with respect to $X$ by $\nabla_XY$. A connection $\nabla$ is termed {\it left-invariant} if it satisfies: $\nabla_{T_{(\cdot)}L_gX}T_{(\cdot)}L_gY=T_{(\cdot)}L_g\nabla_XY$, for all $g\in G$. We also define the {\it torsion} tensor of $\nabla$, $\mathcal{T}:\mathfrak{X}\times \mathfrak{X}\rightarrow \mathfrak{X}$, as $\mathcal T(X,Y):=\nabla_XY-\nabla_YX-[X,Y]$. This term captures the difference between the Lie bracket and the utilized connection. When  $\mathcal T(X,Y)=0$ for all $X,Y\in\mathfrak{X}$, we say that $\nabla$ is {\it symmetric}.

For all left-invariant connections, there exists a bilinear map $\omega:\mathfrak{g}\times\mathfrak{g}\rightarrow\mathfrak{g}$, called the {\it connection function}, such that: $\nabla_{T_eL_gx}T_eL_gy=T_eL_g\omega(x,y)$, with $x,y\in\mathfrak{g}$. Of particular interest in this paper are the {\it Cartan-Schouten} connections. These are determined by $\omega(x,y)=\kappa[x,y]$, where $\kappa=0$, $\kappa=1$, and $\kappa=\frac{1}{2}$ correspond to the (-), (+), and (0) Cartan-Schouten connection, respectively. It can be shown that the (0) connection is a symmetric one \cite{Mahony2002}. 


\textbf{Differential and Hessian operators.} Consider a twice differentiable function $\mathsf f:G\rightarrow\mathbb{R}$, and let $\mathsf x\in T_gG$, $Y\in \mathfrak{X}$ so that $Y:G\rightarrow TG$. The {\it differential} of $\mathsf f$ at $g\in G$ is denoted by $\mathsf D \mathsf f(g):T_gG\rightarrow\mathbb{R}$ and satisfies: $\mathsf x(\mathsf f(g))=\mathsf D\mathsf f(g)(\mathsf x)$. When necessary, we will use a suffix to indicate differentiation with respect to a specific argument (e.g., $\mathsf D_h\mathsf f$ is the differential of $\mathsf f$ with respect to $h$).

The {\it Hessian} operator, $\text{Hess}\mathsf f(g):T_gG\rightarrow T_g^*G$, is a (0,2)-tensor which satisfies the identity: $\dD(\dD\mathsf f(Y))(g)(\mathsf x)=\text{Hess}\mathsf f(g)(\mathsf x)(Y(g))+\dD\mathsf f(g)(\nabla_\mathsf xY)$. In the literature, this mapping is also referred to as the second covariant derivative \cite{Absbook}, or the geometric Hessian \cite{Mahony2002}. When a symmetric connection is used, the Hessian becomes symmetric at all points (i.e., $\text{Hess}\mathsf f(g)=(\text{Hess}\mathsf f(g))^*$, for all $g\in G$). We will often use a superscript to indicate the associated connection function (e.g., $\text{Hess}^{(0)}\mathsf f(g)$ corresponds to the (0) Cartan connection).

Lastly, when we have a vector-valued mapping, $\mathsf f:G\rightarrow V$, the Hessian will be determined for each $\mathsf x,\mathsf y\in T_gG$ by $\text{Hess}\mathsf f(g)(\mathsf x)(\mathsf y)=\sum_{i}(\text{Hess}\mathsf f_i(g)(\mathsf x)(\mathsf y))e_i$. Here, $\{e_i\}$ is a basis on $V$, and $\mathsf f_i(g)$ denotes the $i^{\text{th}}$ component of $\mathsf f(g)$. Similarly, one has $\dD\mathsf f(g)(\mathsf x)=\sum_i(\dD\mathsf f_i(g)(\mathsf x))e_i$.

\textbf{Adjoint representations.} The {\it adjoint representation of $G$}, $\text{Ad}_g:\mathfrak{g}\rightarrow\mathfrak{g}$, is defined as $\text{Ad}_g(\eta):=\mathsf D_h(R_{g^{-1}}L_gh)(e)\cdot\eta$, where $g,h\in G$. For linear groups, one has $\text{Ad}_g\eta=g\eta g^{-1}$. We also define the {\it adjoint representation of $\mathfrak{g}$}, $\text{ad}_\zeta:\mathfrak{g}\rightarrow\mathfrak{g}$, as $\text{ad}_\eta\zeta:=\mathsf D_g(\text{Ad}_g\zeta)(e)\cdot\eta$, for each $\eta,\zeta\in\mathfrak{g}$. This latter operator corresponds to the Lie bracket of $\mathfrak{g}$; that is, $\text{ad}_\eta\zeta=[\eta,\zeta]$.

\textbf{Exponential map.} The {\it exponential map}, $\exp:\mathfrak{g}\rightarrow G$, is a local diffeomorphism defined by: $\exp(\xi):=\gamma(1)$, with $\gamma:\mathbb{R}\rightarrow G$ satisfying $\dot\gamma(0)=\xi$. Its right-trivialized tangent, $\text{dexp}:\mathfrak{g}\times\mathfrak{g}\rightarrow\mathfrak{g}$, is determined so that: $\dD\exp(\xi)\cdot\zeta=T_eR_{\exp(\xi)}\text{dexp}_\xi\zeta$. We also define the {\it logarithm map}, $\log:G\rightarrow\mathfrak{g}$, as the inverse of $\exp(\cdot)$ (i.e., $\log(\exp(\xi))=\xi$). When necessary, we will use a subscript to denote the group with respect to which the above operators are applied. For example, $\exp_G$ denotes the exponential map associated with $G$.

\textbf{Exponential functor.} Given any two vector spaces $K$ and $ V$, let $\mathfrak L(K, V)$ denote the set of linear maps from $K$ to $V$. We define the {\it exponential functor}, $(\cdot)^{ W}$, acting on a linear map $\mathsf g:K\rightarrow V$, such that $\mathsf g^{ W}:\mathfrak L(W,K)\rightarrow\mathfrak L(W,V)$, with $\mathsf g^{W}(\theta):=\mathsf g\circ\theta$, for any vector space $W$ and linear map $\theta\in\mathfrak L(W,K)$.

Finally, we will denote the state and control input of our system at time instant $t=t_k$ by $g^k\in G$ and $u^k\in\mathbb{R}^m$, respectively. The sequence of a state trajectory from $t_0$ to $t_H$ will be given by $\{g^k\}_0^H=\{g^0,g^1,...,g^H\}$. For simplicity, we will often omit the indexes and write $\{g^k\}$. Similarly, a control sequence will be denoted by $\{u^k\}_0^{H-1}$. Furthermore, $\tau_0^H:=\{\{g^k\}_0^H,\{u^k\}_0^{H-1}\}$ will refer to both a state, and its corresponding controls curve.

\section{Deriving the Differential Dynamic Programming algorithm on Lie groups}\label{sec:ddp_der}
DDP applies (sub)optimal control deviations in an iterative fashion, until reaching a local solution of the problem. Towards this goal, we will employ expansion schemes for our objective function and perturbed dynamics on $G$. The details are given below.

\subsection{Problem definition}
Consider the following discrete-time, finite-horizon optimal control problem:
\begin{equation}
\label{optprob}
\begin{split}
&\min_{\{u^k\}_0^{H-1}}\hspace{1.7mm} J(\{g^k\}_0^H,\{u^k\}_0^{H-1})\\
\text{s.t.}\quad g&^{k+1}=f^k(g^k,u^k),\quad g^0 = \bar{g}^0
\end{split}
\end{equation}
with
\begin{equation}
\label{J}
J:=\sum_{k=0}^{H-1} \Lambda^k(g^k,u^k)+F(g^H).
\end{equation}

The time horizon is discretized using a fixed time-step $\Delta t$, yielding $H+1$ distinct instances. $f^k:G\times\mathbb{R}^m\rightarrow G$ is the state transition mapping at time $t=t_k$.  Moreover, $\Lambda^k:G\times\mathbb{R}^m\rightarrow\mathbb{R}$ denotes the running cost, while $F(g^H)\in\mathbb{R}$ is the terminal cost term. Henceforth, we assume that \eqref{optprob} admits a solution, with $J$ and $f^k$ being both twice differentiable for all $k$.

Solving the generic problem in \eqref{optprob} analytically is rarely feasible. Furthermore, obtaining the global minimum numerically can be tedious, especially for high-dimensional systems. Hence, we seek to develop a method that gives tractable solutions, possibly at the expense of global optimality. 

\subsection{Linearization of perturbed state trajectories} \label{sec:ddp_der_lin}
Let $\{\bar{u}^k\}_0^{H-1}$ be a nominal control sequence, and let $\{\bar{g}^k\}_0^H$ denote the coresponding nominal state trajectory. We consider perturbations of the nominal control inputs given by $\{u_\epsilon^k\}_0^{H-1}:=\{\bar{u}^k+\delta u^k\}_0^{H-1}$. Assuming that $\delta u^k\in\mathbb{R}^m$ is small enough for all $k$, the perturbed state trajectory, $\{g_\epsilon^k\}_0^H$, will remain close to the nominal one. Therefore, we can use exponential coordinates to write: $\{g_\epsilon^k\}=\{\bar{g}^k\exp(\zeta^k)\}$, with $\zeta^k\in\mathfrak{g}$.

The derivation of DDP requires a linearization scheme for the perturbation vectors, $\{\zeta^k\}$. In section \ref{sec:zetaexp} we will provide second-order expansions for generic classes of discrete mechanical systems. For now, we will assume that such a scheme is available; that is, we have
\begin{equation}
\label{linzeta}
\begin{split}
&\zeta^{k+1}\approx\Phi^k(\bar{\tau}_0^H)(\zeta^k)+\text B^k(\bar{\tau}_0^H)(\delta u^k)+\frac{1}{2}\big(\Theta^k(\bar{\tau}_0^H)(\zeta^k)(\zeta^k)+\\
&\Gamma^k(\bar{\tau}_0^H)(\zeta^k)(\delta u^k)+\Delta^k(\bar{\tau}_0^H)(\delta u^k)(\zeta^k)+\Xi^k(\bar{\tau}_0^H)(\delta u^k)(\delta u^k)\big),
\end{split}
\end{equation}
where $\Phi^k(\bar{\tau}_0^H):\mathfrak{g}\rightarrow\mathfrak{g}$, $\text B^k(\bar{\tau}_0^H):\mathbb{R}^m\rightarrow\mathfrak{g}$, $\Theta^k(\bar{\tau}_0^H):\mathfrak{g}\times\mathfrak{g}\rightarrow\mathfrak{g}$, $\Gamma^k(\bar{\tau}_0^H):\mathfrak{g}\times\mathbb{R}^m\rightarrow\mathfrak{g}$, $\Delta^k(\bar{\tau}_0^H):\mathbb{R}^m\times\mathfrak{g}\rightarrow\mathfrak{g}$, and $\Xi^k(\bar{\tau}_0^H):\mathbb{R}^m\times\mathbb{R}^m\rightarrow\mathfrak{g}$ are all linear in their arguments. We also require that $\Theta^k_{(i)}(\bar{\tau}_0^H)=(\Theta^k_{(i)}(\bar{\tau}_0^H))^*$, $\Xi^k_{(i)}(\bar{\tau}_0^H)=(\Xi^k_{(i)}(\bar{\tau}_0^H))^*$ and $\Delta^k_{(i)}(\bar{\tau}_0^H)=(\Gamma^k_{(i)}(\bar{\tau}_0^H))^*$, for all $k$, $i$. The subscript here corresponds to a particular component of an operator (i.e., given a basis $\{E_i\}$ in $\mathfrak{g}$, we take $\Theta^k(\bar{\tau}_0^H)(\zeta^k)(\zeta^k)=\sum_i\Theta^k_{(i)}(\bar{\tau}_0^H)(\zeta^k)(\zeta^k)E_i$, etc).

\subsection{Expansion of the $Q$ functions}
Let us first define the discrete {\it value function}, $V^k:G\rightarrow\mathbb{R}$, at time $t=t_k$ as
\begin{equation*}
V^k(g^k):=\min_{\{u^i\}_{k}^{H-1}}\big[\sum_{i=k}^{H-1} \Lambda^i(g^i,u^i)+F(g^H)\big].
\end{equation*}

From Bellman's principle of optimality in discrete time, we have that \cite{ddp, Bertsekas}
\begin{equation}
\label{vf}
V^k(g^k)=\min_{u^k}\big[\Lambda^k(g^k,u^k)+V^{k+1}(g^{k+1})\big].
\end{equation}

We proceed by expanding both sides of eq. \eqref{vf} about a nominal sequence, $\bar\tau_0^H=\{\{\bar g^k\}_0^H,\{\bar u^k\}_0^{H-1}\}$. In particular, we consider Taylor expansions with respect to exponential coordinates. The following set of assumptions is required:
\begin{ass}
(i) The value function, $V^k$, is twice differentiable for each $k=0,1,...,H$, (ii) G is endowed with the (0), (+), or (-) Cartan connection.
\end{ass}
\noindent Now, the perturbed value function can be written as
\begin{equation}
\label{expandv}
\begin{split}
V^k(\bar{g}^k\exp(\kappa\zeta^k))=&V^k(\bar{g}^k)+\kappa\left.\frac{\mathrm d}{\mathrm d s}\right|_{s=0}V^k(\bar{g}^k\exp(s\zeta^k))+\\
\frac{1}{2}\kappa^2&\left.\frac{\mathrm d^2}{\mathrm d s^2}\right|_{s=0}V^k(\bar{g}^k\exp(s\zeta^k))+O(|\kappa|^3)\\
=&V^k(\bar{g}^k)+\kappa\dD V^k(\bar{g}^k)(T_eL_{\bar{g}^k}\zeta^k)+\\
\frac{1}{2}\kappa^2\dD(\dD V^k&(T_eL_{g^k}\zeta^k))(\bar{g}^k)(T_eL_{\bar{g}^k}\zeta^k)+O(|\kappa|^3).
\end{split}
\end{equation}
The first equality is obtained by treating $V^k$ as a function of $\kappa\in\mathbb{R}$, and the second equality has been proven in \cite[eq. 2.12.2]{varadarajanbook} and \cite[page 315]{Mahony2002}. Now, from the definition of the Hessian operator (see section \ref{sec:prel}), the second-order term is equal to $\frac{1}{2}\kappa^2(\text{Hess}V^k(\bar{g}^k)(T_eL_{\bar{g}^k}\zeta^k)(T_eL_{\bar{g}^k}\zeta^k)+T_eL_{\bar{g}^k}\omega(\zeta^k,\zeta^k)(V^k(\bar{g}^k)))$. By using the results from appendix \ref{app:hess}, in conjunction with the skew-symmetry of $\omega(\cdot,\cdot)$, equation \eqref{expandv} becomes
\begin{equation*}
\begin{split}
V^k(g^k_\epsilon)=&V^k(\bar{g}^k)+\mathsf DV^k(\bar{g}^k)\big(T_eL_{\bar{g}^k}\zeta^k\big)+\\
&\frac{1}{2}\text{Hess}^{(0)}V^k(\bar{g}^k)(T_eL_{\bar{g}^k}\zeta^k)(T_eL_{\bar{g}^k}\zeta^k)+O(||\zeta^k||^3),
\end{split}
\end{equation*}
where we have absorbed $\kappa$ into $\zeta^k$. Equivalently:
\begin{equation}
\label{expandvfinal}
\begin{split}
&V^k(g^k_\epsilon)=\\
&V^k(\bar{g}^k)+\mathcal{V}_g^k(\bar{g}^k)\big(\zeta^k\big)+\frac{1}{2}\mathcal{V}_{gg}^k(\bar{g}^k)\big(\zeta^k)(\zeta^k\big)+O(||\zeta^k||^3),
\end{split}
\end{equation}
with
\begin{equation}
\label{trivdif}
\begin{split}
\mathcal{V}^k_g(g^k):=&T_eL_{g^k}^*\circ\mathsf DV^k(g^k)\\
\mathcal{V}_{gg}^k(g^k):=&T_eL_{g^k}^*\circ\text{Hess}^{(0)}V^k(g^k)\circ T_eL_{g^k}.
\end{split}
\end{equation}

The last manipulation transforms the operators $\mathsf DV^k(g^k):T_{g^k}G\rightarrow\mathbb{R}$ and $\text{Hess}^{(0)}V^k(g^k):T_{g^k}G\rightarrow T_{g^k}^*G$, into $\mathcal{V}_g^k(g^k):\mathfrak{g}\rightarrow\mathbb{R}$ and $\mathcal{V}_{gg}^k(g^k):\mathfrak{g}\rightarrow\mathfrak{g}^*$, respectively. In light of this, our algorithm will be derived by solely using operations on $(\mathfrak{g}^*,\mathfrak{g})$. This will allow us, for example, to backpropagate the (trivialized) differential and Hessian of the value function along nominal trajectories. From a computational standpoint, by defining a basis for $\mathfrak{g}$ and its dual, we will be able to implement all steps through standard matrix/vector products.

Now, the running cost in \eqref{vf} will be similarly expanded as follows:
\begin{equation}
\label{lexp}
\begin{split}
\Lambda^k&(g^k_\epsilon,u^k_\epsilon)\approx\Lambda^k(\bar{g}^k,\bar{u}^k)+\langle\ell_g^k(\bar{g}^k,\bar{u}^k),\zeta^k\rangle+\langle\ell_u^k(\bar{g}^k,\bar{u}^k),\delta u^k\rangle+\\
&\frac{1}{2}\bigg(\langle\ell^k_{gg}(\bar{g}^k,\bar{u}^k)(\zeta^k),\zeta^k\rangle+\langle\ell^k_{gu}(\bar{g}^k,\bar{u}^k)(\zeta^k),\delta u^k\rangle+\\
&\hspace{7mm}\langle\ell^k_{ug}(\bar{g}^k,\bar{u}^k)(\delta u^k),\zeta^k\rangle+\langle\ell^k_{uu}(\bar{g}^k,\bar{u}^k)(\delta u^k),\delta u^k\rangle\bigg),
\end{split}
\end{equation}
with
\begin{equation}
\label{leftell}
\begin{split}
&\scalebox{0.93}{$\ell_g^k(g^k,u^k):=T_eL_{g^k}^*\circ \mathsf D_g\Lambda^k(g^k,u^k),\quad\ell_u^k(g^k,u^k):=\mathsf D_u\Lambda^k(g^k,u^k)$},\\
&\scalebox{0.93}{$\ell_{uu}^k(g^k,u^k):=\mathsf D^2_u\Lambda^k(g^k,u^k),\quad\ell^k_{gu}(g^k,u^k):=\mathsf D_g\mathsf D_u\Lambda^k(g^k,u^k)\circ T_eL_{g^k}$},\\
&\scalebox{0.93}{$\ell^k_{ug}(g^k,u^k):=T_eL_{g^k}^*\circ\mathsf D_u\mathsf D_g\Lambda^k(g^k,u^k)$},\\
&\scalebox{0.93}{$\ell^k_{gg}(g^k,u^k):=T_eL_{g^k}^*\circ\text{Hess}_g^{(0)}\Lambda^k(g^k,u^k)\circ T_eL_{g^k}$}.
\end{split}
\end{equation}
Note that $\mathsf D_u^2$ denotes the standard Euclidean Hessian with respect to $u$. Next, we rewrite equation \eqref{expandvfinal} for $t_{k+1}$:
\begin{equation}
\label{expandvfinal1}
\begin{split}
V^{k+1}(g^{k+1}_\epsilon)\approx &V^{k+1}(\bar{g}^{k+1})+\mathcal{V}_g^{k+1}(\bar{g}^{k+1})\big(\zeta^{k+1}\big)+\\
&\frac{1}{2}\mathcal{V}_{gg}^{k+1}(\bar{g}^{k+1})\big(\zeta^{k+1})(\zeta^{k+1}\big).
\end{split}
\end{equation}
Using the linearization scheme of $\zeta^{k+1}$ from \eqref{linzeta} and the bilinearity of the Hessian operator, we rewrite eq. \eqref{expandvfinal1} as
\begin{equation}
\label{vzeta}
\begin{split}
V^{k+1}(&g^{k+1}_\epsilon)\approx V^{k+1}+\langle(\Phi^k)^*\circ\mathcal{V}_{g}^{k+1},\zeta^k\rangle+\langle(\text B^k)^*\circ\mathcal{V}_{g}^{k+1},\delta u^k\rangle\\
&+\frac{1}{2}\big(\langle(\Phi^k)^*\circ\mathcal{V}_{gg}^{k+1}\circ\Phi^k(\zeta^k)+\mathcal{V}_{g}^{k+1}\circ\Theta^k(\zeta^k),\zeta^k\rangle+\\
&\hspace{7mm}\langle\big[(\Phi^k)^*\circ\mathcal{V}_{gg}^{k+1}\circ \text B^k(\delta u^k)+\mathcal{V}_{g}^{k+1}\circ\Delta^k(\delta u^k),\zeta^k\rangle+\\
&\hspace{7mm}\langle\big[(\text B^k)^*\circ\mathcal{V}_{gg}^{k+1}\circ\Phi^k(\zeta^k)+\mathcal{V}_{g}^{k+1}\circ\Gamma^k(\zeta^k),\delta u^k\rangle+\\
&\hspace{7mm}\langle\big[(\text B^k)^*\circ\mathcal{V}_{gg}^{k+1}\circ \text B^k(\delta u^k)+\mathcal{V}_{g}^{k+1}\circ\Xi^k(\delta u^k),\delta u^k\rangle\big),
\end{split}
\end{equation}
where higher order terms have been ignored. Moreover, it is implied that the right hand side of \eqref{vzeta} is evaluated at $\bar{\tau}_0^H$. For compactness, let us define the {\it Q function} as follows
\begin{equation}
\label{q}
Q^k(g^k,u^k):=\Lambda^k(g^k,u^k)+V^{k+1}(g^{k+1}).
\end{equation}
From eqs. \eqref{lexp} and \eqref{vzeta}, the perturbed $Q$ function can be approximated by
\begin{equation}
\label{Qexp}
\begin{split}
Q^k(g^k_\epsilon&,u^k_\epsilon)\approx Q_0^k+\langle Q_g^k,\zeta^k\rangle+\langle Q_u^k,\delta u^k\rangle+\frac{1}{2}\big(\langle Q_{gg}^k(\zeta^k),\zeta^k\rangle+\\
&\langle Q_{ug}^k(\delta u^k),\zeta^k\rangle+\langle Q_{gu}^k(\zeta^k),\delta u^k\rangle+\langle Q_{uu}^k(\delta u^k),\delta u^k\rangle\big),
\end{split}
\end{equation}
so that
\begin{equation}
\label{Qfunctions}
\begin{split}
Q_0^k:=&\Lambda^k+V^{k+1},\\ Q_g^k:=&\ell_g^k+(\Phi^k)^*\circ\mathcal{V}_{g}^{k+1},\quad Q_u^k:=\ell_u^k+(\text B^k)^*\circ\mathcal{V}_{g}^{k+1},\\
Q_{gg}^k:=&\ell_{gg}^k+(\Phi^k)^*\circ\mathcal{V}_{gg}^{k+1}\circ\Phi^k+(\mathcal{V}_{g}^{k+1})^\mathfrak{g}\circ\Theta^k,\\
Q_{gu}^k:=&\ell_{gu}^k+(\text B^k)^*\circ\mathcal{V}_{gg}^{k+1}\circ\Phi^k+(\mathcal{V}_{g}^{k+1})^{\mathbb{R}^m}\circ\Gamma^k,\\
Q_{ug}^k:=&\ell_{ug}^k+(\Phi^k)^*\circ\mathcal{V}_{gg}^{k+1}\circ \text B^k+(\mathcal{V}_{g}^{k+1})^\mathfrak{g}\circ\Delta^k,\\
Q_{uu}^k:=&\ell_{uu}^k+(\text B^k)^*\circ\mathcal{V}_{gg}^{k+1}\circ \text B^k+(\mathcal{V}_{g}^{k+1})^{\mathbb{R}^m}\circ\Xi^k.
\end{split}
\end{equation}
Above we have used the exponential functor to drop all arguments from $\Theta^k$, $\Gamma^k$, $\Delta^k$ and $\Xi^k$.

\subsection{Computing the (sub)optimal control deviations}
From the definition of the $Q$ function in \eqref{q}, equation \eqref{vf} can be transformed into
\begin{equation}
\label{perturbedvf}
V^k(g_\epsilon^k)=\min_{\delta u^k}\big[Q^k(g^k_\epsilon,u^k_\epsilon)\big].
\end{equation}
Notice that we are optimizing with respect to $\delta u^k$, since the new controls are determined as $u^k=\bar{u}^k+\delta u^k$, with $\bar{u}^k$ fixed for all $k$. By utilizing the quadratic expansion in \eqref{Qexp}, we can explicitly perform the minimization on the right-hand side of \eqref{perturbedvf}. Since the natural pairing is bilinear, one obtains the (locally) optimal control deviations:
\begin{equation}
\label{dumin}
\delta u^k_\star=-(Q_{uu}^k)^{-1}\circ Q_u^k-(Q_{uu}^k)^{-1}\circ Q_{gu}^k(\zeta^k).
\end{equation}
For this expression we have used the symmetry of $\text{Hess}^{(0)}$ (and, therefore, of $\mathcal{V}_{gg}$). This implies that $Q_{gg}^k=(Q_{gg}^k)^*$, $Q_{uu}^k=(Q_{uu}^k)^*$, and since $\mathsf D_g\mathsf D_u\Lambda^k=(\mathsf D_u\mathsf D_g\Lambda^k)^*$, then $Q_{gu}^k=(Q_{ug}^k)^*$.

Mimicking the approach in \cite{liao000} for the Euclidean version of DDP, we add an external parameter, $\gamma\in(0,1]$, in the controls update. Intuitively, this will allow us to generate descent directions even when the utilized quadratic expansions (e.g., eqs. \eqref{expandvfinal}, \eqref{Qexp}) do not fully capture the nature of the problem. It will be shown in section \ref{sec:conv}, that $\gamma$ plays a key role in the convergence of our algorithm. Hence, we will use during implementation
\begin{equation}
\label{dustar}
\delta u^k_\star=-\gamma(Q_{uu}^k)^{-1}\circ Q_u^k-(Q_{uu}^k)^{-1}\circ Q_{gu}^k(\zeta^k),
\end{equation}
for each $k=0,...,H-1$.

\subsection{Backpropagation schemes for $\mathcal{V}_{g}$ and $\mathcal{V}_{gg}$}
Observe that computing the (sub)optimal control updates requires knowledge of $\mathcal{V}_{g}$ and $\mathcal{V}_{gg}$ on the nominal state trajectory. To this end, we incorporate the quadratic expansions of $V^k$ and $Q^k$ (eqs. \eqref{expandvfinal}, \eqref{Qexp}), as well as $\delta u^k_\star$ into \eqref{perturbedvf} to get\footnote{The min operator in \eqref{perturbedvf} drops after plugging $\delta u^k_\star$ from \eqref{dumin}.}
\begin{equation*}
\begin{split}
&V^k+\langle\mathcal{V}_g^k,\zeta^k\rangle+\frac{1}{2}\langle\mathcal{V}_{gg}^k(\zeta^k),\zeta^k\rangle+O(||\zeta^k||^3)=\\
&Q_0^k+\langle Q_g^k,\zeta^k\rangle+\langle Q_u^k,-(Q_{uu}^k)^{-1}\circ Q_u^k-(Q_{uu}^k)^{-1}\circ Q_{gu}^k(\zeta^k)\rangle+\\
&\langle Q_{gu}^k(\zeta^k),-(Q_{uu}^k)^{-1}\circ Q_u^k-(Q_{uu}^k)^{-1}\circ Q_{gu}^k(\zeta^k)\rangle+\\
&\frac{1}{2}\langle Q_{gg}^k(\zeta^k),\zeta^k\rangle+\frac{1}{2}\langle Q_{uu}^k\circ(Q_{uu}^k)^{-1}\circ Q_u^k,(Q_{uu}^k)^{-1}\circ Q_u^k\rangle+\\
&\langle Q_{uu}^k\circ(Q_{uu}^k)^{-1}\circ Q_u^k,(Q_{uu}^k)^{-1}\circ Q_{gu}^k(\zeta^k)\rangle+\\
&\frac{1}{2}\langle Q_{uu}^k\circ(Q_{uu}^k)^{-1}\circ Q_{gu}^k(\zeta^k),(Q_{uu}^k)^{-1}\circ Q_{gu}^k(\zeta^k)\rangle.
\end{split}
\end{equation*}
Since the above result holds for arbitrary $\zeta^k$, we can match the first and second-order terms. After a simple manipulation, we obtain the following expressions
\begin{equation}
\label{backv}
\begin{split}
\mathcal{V}_g^k&=Q^k_g-Q^k_{ug}\circ(Q_{uu}^k)^{-1}\circ Q^k_u,\\
\mathcal{V}_{gg}^k&=Q^k_{gg}-Q^k_{ug}\circ(Q_{uu}^k)^{-1}\circ Q^k_{gu}.
\end{split}
\end{equation}
Recall that all quantities above are evaluated at $\bar\tau_0^H$, with the right-hand sides depending on $\mathcal{V}_g^{k+1}$ and $\mathcal{V}_{gg}^{k+1}$. The final condition for this backpropagation scheme is given by
\begin{equation}
\label{vfinal}
\begin{split}
\mathcal{V}^H_g(\bar{g}^H)=&T_eL_{\bar{g}^H}^*\circ\mathsf DF(\bar{g}^H),\\
\mathcal{V}_{gg}^H(\bar{g}^H)=&T_eL_{\bar{g}^H}^*\circ\text{Hess}^{(0)}F(\bar{g}^H)\circ T_eL_{\bar{g}^H},
\end{split}
\end{equation}
where $F$ denotes the terminal cost term. 

We conclude this section by summarizing our framework in Algorithm \ref{alg:ddp}. Moreover, Figure \ref{fig:ddp} shows an illustration of the applied steps.
\begin{algorithm}[h]
	\SetAlgoLined
	\KwData{Dynamics and cost functions of \eqref{optprob}, nominal control sequence $\{\bar{u}^k\}_{0}^{H-1}$, fixed initial state $\bar{g}^0$;}
	\For{$(k=0;$ $k\leq H-1;$ $k++)$}{
	Get the next nominal state $\bar{g}^{k+1}\leftarrow f^k(\bar{g}^k,\bar{u}^k)$;}
	Compute the nominal cost, $\bar{J}$, from \eqref{J};\\
	\Repeat{\emph{convergence}}{
	Calculate $\mathcal{V}^H_g$ and $\mathcal{V}_{gg}^H$ from \eqref{vfinal};\\
		\For{$(k=H-1;$ $k\geq0;$ $k--)$}{
		Determine $\Phi^k$, $\text{B}^k$, $\Theta^k$, $\Gamma^k$, $\Delta^k$ and $\Xi^k$ for the linearization scheme in \eqref{linzeta};\\
		Find the trivialized derivatives of $Q^k$ from \eqref{Qfunctions};\label{stepQ}\\
		Backpropagate $\mathcal{V}^k_g$, $\mathcal{V}_{gg}^k$ through \eqref{backv};}
		Set the line-search parameter, $\gamma\leftarrow1$;\\
		\Repeat{$J-\bar{J}\leq0$}{
		Set $\zeta^0\leftarrow0$, $g^0\leftarrow\bar{g}^0$;\\
		\For{$(k=0;$ $k\leq H-1;$ $k++)$}{
		Compute a new control $u^k\leftarrow\bar{u}^k+\delta u^k_\star$, with $\delta u^k_\star$ defined in \eqref{dustar};\\
		Get the next state $g^{k+1}\leftarrow f^k(g^k,u^k)$;\\
		Find $\zeta^{k+1}\leftarrow\log\big((\bar{g}^{k+1})^{-1}g^{k+1}\big)$;}
		Given the obtained state trajectory, $\{g^k\}_0^H$, and control sequence, $\{u^k\}_0^{H-1}$, compute the new cost, $J$;\\
		Set $\gamma\leftarrow\mathsf h\gamma$, with $\mathsf h\in(0,1)$;\label{steph}}
		Update $\{\bar{u}^k\}_0^{H-1}\leftarrow \{u^k\}_0^{H-1}$, $\{\bar{g}^k\}_0^{H}\leftarrow\{g^k\}_0^{H}$, $\bar{J}\leftarrow J$;
	}
	\caption{DDP on Lie groups}
	\label{alg:ddp}
\end{algorithm}

\begin{figure}
	\centering
	\subfigure[]{
		\includegraphics[width=0.233\textwidth]{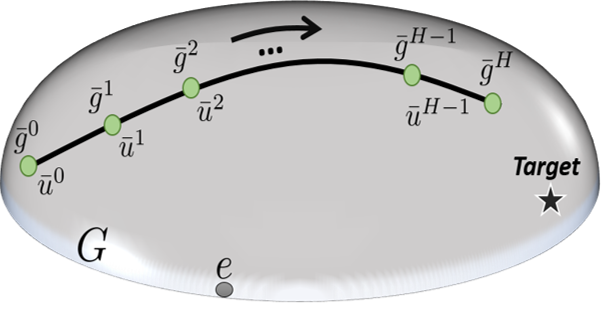}}
	\subfigure[]{
		\includegraphics[width=0.233\textwidth]{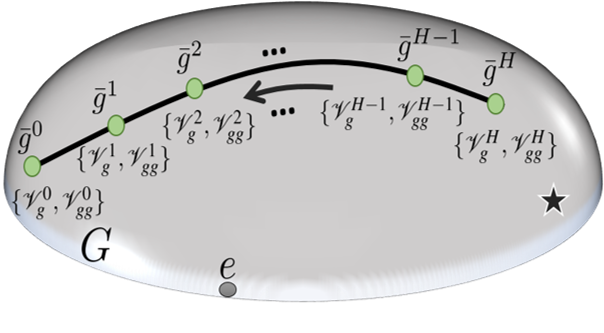}}
	\subfigure[]{
		\includegraphics[width=0.233\textwidth]{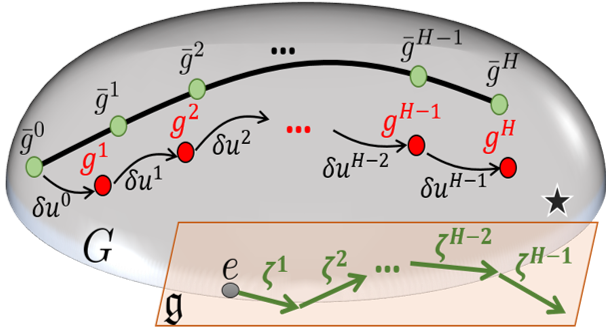}}
	\subfigure[]{
		\includegraphics[width=0.233\textwidth]{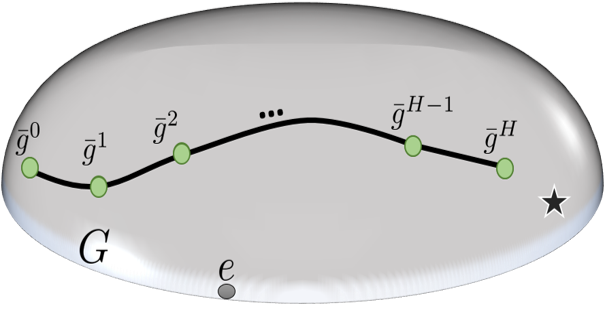}}
	\caption{Illustration of Differential Dynamic Programming on Lie groups: (a) Given a nominal control sequence, the corresponding trajectory is generated on the configuration manifold, (b) The (trivialized) derivatives of the value function are backpropagated along the nominal trajectory, (c) Control updates are determined that yield a new state and control sequence. This requires computing the linearized state perturbations on the Lie algebra, (d) The updated sequence is treated as the nominal one, and the procedure is repeated until convergence.}\label{fig:ddp}
\end{figure}

\section{Second-order expansions for discrete mechanical systems} \label{sec:zetaexp}
In this section we address linearization methods for mechanical systems. An analogous scheme has been mentioned in \cite{kobilarovddplie} and \cite{liegroupddp_estimation}; however, these works employed a first-order approximation, for which they did not provide a mathematical proof. Our contribution lies in deriving the second-order expansion in \eqref{linzeta} for generic classes of discrete dynamics. To this end, we will make use of the Baker-Campbell-Hausdorff (BCH) formula which links a Lie group to its Lie algebra. As a byproduct, we will also include an alternative approach for obtaining \eqref{linzeta} up to linear terms only.

The state of most physical systems lies on the tangent bundle of a Lie group, $G'$. Since we can always find an isomorphism between $G'\times\mathfrak g'$ and $TG'$ \cite{Gallier_notes}, we typically decompose our state as: $g:=(\chi,\xi)\in G'\times\mathfrak{g}'$ (i.e., the configuration space becomes $G=G'\times\mathfrak{g}'$). One can view $\chi$ as the {\it pose} of the system, and $\xi$ as the {\it body-fixed velocity}.

The continuous equations of motion for a fairly large class of mechanical systems is given by:
\begin{equation}
\label{conteqs1}
\begin{split}
\dot{\chi}(t)=&\chi(t)\xi(t),\\
\dot{\xi}(t)=&F_\xi(\chi(t),\xi(t),u(t)).
\end{split}
\end{equation}
Such expressions can be obtained by employing Hamilton's principle (see, e.g., \cite{hpintegrators, kobpontr}). Since we are interested in discrete control algorithms, we will next discretize eqs. \eqref{conteqs1}.

{\it Explicit transition dynamics:} The simplest scheme is given by the forward Euler method:
\begin{subequations}\label{disceq}
\begin{align}
\chi^{k+1}&=\chi^k\exp_{G'}(\Delta t\xi^k),\label{receq}\\
\xi^{k+1}&=f_\xi^k(g^k,u^k),\label{dynamicseq}
\end{align}
\end{subequations}
where \eqref{receq} is the {\it reconstruction equation} and $f_\xi^k:=\xi^k+\Delta tF_\xi(\chi^k,\xi^k,u^k)$. First, we will work with \eqref{receq} and \eqref{dynamicseq}. We will see, however, that a similar approach can be used for different discretizations as well (e.g., backward or implicit integrators).

Now, let $\{\bar{g}^k\}$, $\{\bar{u}^k\}$ be a nominal state trajectory and control sequence, respectively. Define also the state perturbation vectors as $\zeta^k:=(\eta^k,\delta\xi^k)\in\mathfrak{g}$, which correspond to control deviations $\delta u^k\in\mathbb{R}^m$. Using the same reasoning as in section \ref{sec:ddp_der_lin}, the perturbed state, $g_\epsilon^k:=(\chi^k_\epsilon,\xi^k_\epsilon)$, will be equivalent to
\begin{equation}
\label{pertzeta}
g^k_\epsilon=\bar{g}^k\exp_G(\zeta^k)=\big(\bar{\chi}^k\exp_{G'}(\eta^k),\bar{\xi}^k+\delta\xi^k\big).
\end{equation}
Note that $\mathfrak{g}'$ is a flat space, and hence the left translation map is given by vector addition, while $\exp_{\mathfrak{g}'}$ is equal to the identity operator. 

We begin by expanding $\eta^{k+1}$ with respect to $\eta^k$ and $\delta\xi^k$. Observe that the perturbed trajectory, $\{g_\epsilon^k\}$, will satisfy the kinematics equation \eqref{receq}. Thus, one has

\[
\begin{split}
\eqref{receq}\Rightarrow\chi^{k+1}_\epsilon&=\chi^k_\epsilon\exp_{G'}(\Delta t\xi^k_\epsilon)\\
\overset{\eqref{pertzeta}}{\Rightarrow}\bar{\chi}^{k+1}\exp_{G'}(\eta^{k+1})&=\bar{\chi}^k\exp_{G'}(\eta^k)\exp_{G'}(\Delta t\xi^k_\epsilon)\\
\overset{\eqref{receq}}{\Rightarrow}\bar{\chi}^k\exp_{G'}(\Delta t\bar{\xi}^k)\exp_{G'}(\eta^{k+1})&=\bar{\chi}^k\exp_{G'}(\eta^k)\exp_{G'}(\Delta t\xi^k_\epsilon),
\end{split}
\]
or, equivalently
\begin{equation}
\label{etaexpression}
\eta^{k+1}=\log_{G'}\big(\exp_{G'}(-\Delta t\bar{\xi}^k)\exp_{G'}(\eta^k)\exp_{G'}(\Delta t\bar{\xi}^k+\Delta t\delta\xi^k)\big).
\end{equation}

We proceed by applying the Baker-Campbell-Hausdorff formula repeatedly on the right-hand side of \eqref{etaexpression}. Details about the form and technical assumptions of the BCH series are given in appendix \ref{sec:bch}. For the remainder of this section we will simply write $\exp$ ($\log$, resp.) instead of $\exp_{G'}$ ($\log_{G'}$, resp.). Moreover, we will be neglecting terms of order $O(\Delta t^3)$, since one typically has $0<\Delta t\leq0.1$.

Equations \eqref{bch} and \eqref{bchlinear} imply that
\begin{equation}
\label{bchz0_0}
\exp(\eta^k)\exp(\Delta t\bar{\xi}^k+\Delta t\delta\xi^k)=\exp(\Delta t\bar{\xi}^k+Z),
\end{equation}
where\footnote{With a slight abuse of notation, $O(||(\mathsf x,\mathsf y)||^{\mathrm n})$ is equivalent to $O(||\mathsf x||^{\mathrm n_x}||\mathsf y||^{\mathrm n_y})$, with $\mathrm n_x+\mathrm n_y=\mathrm n$.}
\begin{equation*}
\begin{split}
Z:=&\Delta t\delta\xi^k+\text{dexp}_{(\Delta t\bar{\xi}^k+\Delta t\delta\xi^k)}^{-1}(\eta^k)+\frac{1}{12}\text{ad}_{\eta^k}^2(\Delta t\bar{\xi}^k)+\\
&\frac{1}{24}\text{ad}_{\eta^k}\text{ad}_{(\Delta t\bar{\xi}^k)}^2(\eta^k)+O(\Delta t^3)+O(||(\eta^k,\delta\xi^k)||^3).
\end{split}
\end{equation*}
From eq. \eqref{dexp}, we can rewrite $Z\in\mathfrak{g}'$ as
\begin{equation}
\label{bchz0_1}
\begin{split}
&Z=\Delta t\delta\xi^k+\text{dexp}_{(\Delta t\bar{\xi}^k)}^{-1}(\eta^k)-\frac{\Delta t}{2}\text{ad}_{\delta\xi^k}(\eta^k)+\\
&\frac{\Delta t}{12}\text{ad}_{\eta^k}^2\bar{\xi}^k+\frac{\Delta t^2}{12}\big[\frac{1}{2}\text{ad}_{\eta^k}\text{ad}_{\bar{\xi}^k}^2(\eta^k)+\text{ad}_{\bar{\xi}^k}\text{ad}_{\delta\xi^k}(\eta^k)+\\
&\text{ad}_{\delta\xi^k}\text{ad}_{\bar{\xi}^k}(\eta^k)\big]+O(\Delta t^3)+O(||(\eta^k,\delta\xi^k)||^3).
\end{split}
\end{equation}
Now, plugging \eqref{bchz0_0} into \eqref{etaexpression} yields
\begin{equation}
\label{eta1}
\eta^{k+1}=\log\big(\exp(-\Delta t\bar{\xi}^k)\exp(\Delta t\bar{\xi}^k+Z)\big).
\end{equation}
It remains to apply again the BCH formula on \eqref{eta1}. Due to the structure of the right-hand side, we can utilize the result from \eqref{bch1}. Specifically
\begin{equation*}
\eta^{k+1}=\text{dexp}_{(-\Delta t\bar{\xi}^k)}(Z)+O(||Z||^2).
\end{equation*}
Finally, since $\dexp_{(\cdot)}(\cdot)$ is linear in its second argument, equations \eqref{bchz0_1}, \eqref{bch} and \eqref{dexp} imply that
\begin{equation}
\begin{split}
\label{etanew}
&\eta^{k+1}=\text{Ad}_{\exp(-\Delta t\bar{\xi}^k)}(\eta^k)+\Delta t\text{dexp}_{(-\Delta t\bar{\xi}^k)}(\delta\xi^k)+\\
&\text{dexp}_{(-\Delta t\bar{\xi}^k)}\big[-\frac{\Delta t}{2}\text{ad}_{\delta\xi^k}(\eta^k)+\frac{\Delta t}{12}\text{ad}_{\eta^k}^2(\bar{\xi}^k)\big]+\\
&\frac{\Delta t^2}{12}\big[\frac{1}{2}\text{ad}_{\eta^k}\text{ad}_{\bar{\xi}^k}^2(\eta^k)+\text{ad}_{\bar{\xi}^k}\text{ad}_{\delta\xi^k}(\eta^k)+2\text{ad}_{\delta\xi^k}\text{ad}_{\bar{\xi}^k}(\eta^k)+\\
&\ad_{\eta^k}\ad_{\bar{\xi}^k}(\delta\xi^k)+\frac{1}{2}\ad_{\eta^k}^2(\bar{\xi}^k)\big]-\frac{\Delta t}{12}\ad^2_{\dexp^{-1}_{(\Delta t\bar{\xi}^k)}(\eta^k)}(\bar{\xi}^k)+\\
&O(\Delta t^3)+O(||(\eta^k,\delta\xi^k)||^3).
\end{split}
\end{equation}
For the term that is linear in $\eta^k$, we have used the property: $\text{dexp}_\rho(\lambda)=\text{Ad}_{\exp(\rho)}\big(\text{dexp}_{(-\rho)}(\lambda)\big)$ (see \cite{hpintegrators, iserles} for proof). Recall also that $\ad_{(\cdot)}(\cdot)$ and $\dexp_{(\cdot)}(\cdot)$ above are associated with the Lie algebra $\mathfrak{g}'$.

In appendix \ref{order1eta} we show an alternative method for deriving \eqref{etanew} up to first order only. The corresponding proof is based on standard multivariable calculus and group operations.
\begin{remark}
The BCH expansion does not depend on the selected affine connection.
\end{remark}

Regarding the transition dynamics \eqref{dynamicseq}, one can use the same approach as in \eqref{expandv} - \eqref{trivdif}. Let $\{E_i\}_1^n$ denote a basis for $\mathfrak{g}'$, so that $\xi^k=\sum_{i=1}^nE_i\xi^k_i$ and $f_\xi^k=\sum_{i=1}^nE_i(f_\xi^k)_i$. Then
\begin{equation}
\label{expandxi}
\begin{split}
&\delta\xi^{k+1}_i\approx T_eL_{\bar{\chi}^k}^*\circ\mathsf D_\chi(f_\xi^k)_i(\bar{g}^k,\bar{u}^k)(\eta^k)+\\
&\mathsf D_\xi(f_\xi^k)_i(\bar{g}^k,\bar{u}^k)(\delta\xi^k)+\mathsf D_u(f_\xi^k)_i(\bar{g}^k,\bar{u}^k)(\delta u^k)+\\
&\frac{1}{2}\big[T_eL_{\bar{\chi}^k}^*\circ\text{Hess}^{(0)}_\chi(f_\xi^k)_i(\bar{g}^k,\bar{u}^k)\circ T_eL_{\bar{\chi}^k}(\eta^k,\eta^k)+\\
&\mathsf D^2_\xi(f_\xi^k)_i(\bar{g}^k,\bar{u}^k)(\delta\xi^k,\delta\xi^k)+\mathsf D_u^2(f_\xi^k)_i(\bar{g}^k,\bar{u}^k)(\delta u^k,\delta u^k)+\\
&2T_eL_{\bar{\chi}^k}^*\circ\mathsf D_\xi \dD_\chi(f_\xi^k)_i(\bar{g}^k,\bar{u}^k)(\delta\xi^k)(\eta^k)+\\
&2T_eL_{\bar{\chi}^k}^*\circ\mathsf D_u \dD_\chi(f_\xi^k)_i(\bar{g}^k,\bar{u}^k)(\delta u^k)(\eta^k)+\\
&2\mathsf D_u\mathsf D_\xi(f_\xi^k)_i(\bar{g}^k,\bar{u}^k)(\delta u^k)(\delta\xi^k)\big],
\end{split}
\end{equation}

{\it Matrix/vector representation:} For ease of implementation, equations \eqref{etanew} and \eqref{expandxi} can be used to write \eqref{linzeta} in matrix/vector form. We give the details in appendix \ref{app:matrixform}.

{\it Implicit transition dynamics:}
Frequently, implicit integrators are employed to propagate dynamics forward in time. These achieve improved numerical performance compared to explicit discretization methods \cite{haier, leethesis}. In particular, the discrete Lagrange-d'Alembert-Pontryagin principle expresses the transition dynamics as follows \cite{kobilarovdiscrete}
\begin{subequations}\label{disceq1}
	\begin{align}
	\chi^{k+1}&=\chi^k\exp(\Delta t\xi^k),\label{receq1}\\
	\mathrm f_\xi(\chi^k,\xi^k&,\chi^{k+1},\xi^{k+1},u^k)=0,\label{dynamicseq1}
	\end{align}
\end{subequations}
with $\mathrm f_\xi(\chi^k,\xi^k,\chi^{k+1},\xi^{k+1},u^k)\in\mathfrak{g}'$ properly defined. The forward variational Euler method constitutes one example of this class (see \cite{hpintegrators} for more details and a derivation of $\mathrm f_\xi$). In this case, we get the updated body-fixed velocity by solving \eqref{dynamicseq1} for $\xi^{k+1}$, through a Newton-like method. Thus, the mapping $\xi^{k+1}=f^{k}_\xi(g^k,u^k)$ is here implicitly determined by \eqref{receq1}, \eqref{dynamicseq1}.

One can define the same expansion as in \eqref{expandxi}; however, the required derivatives will now be obtained through implicit differentiation of \eqref{dynamicseq1}. Specifically, assuming $(\dD_{\xi^{k+1}}\mathrm f_\xi(\bar g^k,\bar g^{k+1},\bar u^k))^{-1}$ exists, the chain rule gives
\begin{equation}
\label{firstimp}\dD_{\chi^k}\bar{\mathrm f}_\xi(\mathsf x)=0\Rightarrow\mathsf D_{\chi^k} f_\xi^k(\mathsf x)=-(\mathsf D_{\xi^{k+1}} \bar{\mathrm f}_\xi)^{-1}\circ\dD_{\chi^k}\bar{\mathrm f}_\xi(\mathsf x),
\end{equation}
with $\bar{\mathrm f}_\xi(\chi^k,\xi^k,\xi^{k+1},u^k):=\mathrm f_\xi(\chi^k,\xi^k,\chi^k\exp(\Delta t\xi^k),\xi^{k+1},u^k)$, and $\mathsf x=\chi^k\eta\in T_{\chi^k}G'$ being an arbitrary vector. Regarding the Hessians computation, we will apply the differential operator repeatedly on $\bar{\mathrm f}_\xi$. For instance, \eqref{firstimp} implies
\begin{equation*}
\begin{split}
&\dD_{\chi^k}\big(\dD_{\chi^k}\bar{\mathrm f}_\xi(\mathsf x)+\sum_i\langle\mathsf D_{\xi^{k+1}} (\bar{\mathrm f}_\xi)_i, \dD_{\chi^k} f_\xi^k(\mathsf x)\rangle E_i\big)(\mathsf x)=0\Rightarrow\\
&\text{Hess}_{\chi^k}f_\xi^k(\mathsf x)(\mathsf x)=-(\mathsf D_{\xi^{k+1}} \bar{\mathrm f}_\xi)^{-1}\circ\big(\text{Hess}_{\chi^k}\bar{\mathrm f}_\xi(\mathsf x)(\mathsf x)+\\
&\sum_i\langle(\mathsf D_{\chi^k} f_\xi^k)^*\circ(\dD_{\chi^k}\dD_{\xi^{k+1}}(\bar{\mathrm f}_\xi)_i+\dD_{\xi^{k+1}}^2(\bar{\mathrm f}_\xi)_i\circ\mathsf D_{\chi^k} f_\xi^k)(\mathsf x),\mathsf x\rangle E_i\\
&+\dD_{\xi^{k+1}}\dD_{\chi^k}\bar{\mathrm f}_\xi(\mathsf D_{\chi^k} f_\xi^k(\mathsf x))(\mathsf x)\big),
\end{split}
\end{equation*}
where we have used the Hessian operator definition, the chain rule, the linearity of $\langle\cdot,\cdot\rangle$, and the skew symmetry of $\omega(\cdot,\cdot)$. It is also assumed that all quantities above are evaluated on the nominal state/control sequence. The remaining expressions can be obtained in the same manner. Notice that the Hessians rely explicitly on the first-order terms, $\dD_{(\cdot)}f_\xi^k$.

We conclude this section by noting that one could first update the body-fixed velocities, and then determine the pose of the system (e.g., as in the backward variational Euler method - see \cite{hpintegrators}). In that case, the reconstruction equation would read $\chi^{k+1}=\chi^k\exp(\Delta t\xi^{k+1})$, and, thus, \eqref{etanew} would contain derivatives of $\xi^{k+1}$ with respect to $g^k$ and $u^k$. We skip the details for compactness.

\section{Convergence analysis}\label{sec:conv}
We study the convergence properties of the algorithm developed in section \ref{sec:ddp_der}. An analogous analysis for the Euclidean case can be found in \cite{liao000} and \cite{ddp0}. Therein, the authors showed that, under some mild conditions, the original DDP method will always converge to a solution. Their work is limited, though, to optimal control problems with terminal costs only. In this section we provide a convergence analysis that deals with the generic problem in \eqref{optprob}, and handles systems evolving on Lie groups.

In what follows, let $\{\delta u^k_\star\}_{0}^{H-1}$ be the (sub)optimal control updates given by \eqref{dustar}. Let also $\bar{U}\in\mathbb{R}^{m(H-1)}$ and $\delta U_\star\in\mathbb{R}^{m(H-1)}$ denote an entire sequence of nominal inputs and updates, respectively; that is, 
$\bar U:=\big((\bar u^0)^\top,...,(\bar u^{H-1})^\top\big)^\top$ and $\delta U_\star:=\big((\delta u^0_\star)^\top,...,(\delta u^{H-1}_\star)^\top\big)^\top$. Moreover, recall from \eqref{linzeta} that given (small enough) control perturbations $\{\bar{u}^k+\upsilon^k\}$, we can define the perturbed state trajectory as $\{\bar{g}^k\exp(\zeta^k)\}$, with
\begin{equation}
\label{linzeta1}
\zeta^{k+1}=\Phi^k(\zeta^k)+\text B^k( \upsilon^k)+O(||(\zeta^k, \upsilon^k)||^2),
\end{equation}
for each $k=0,1,...,H-1$.

We begin by stating one lemma and one set of assumptions that will be used in our analysis.
\begin{lemma}\label{lemm:psi}
	Define $\psi^k\in\mathfrak{g}^*$ as
	\begin{equation}
	\begin{split}
	\label{psi}
	\psi^k: &= \ell_g^k(\bar{g}^k,\bar{u}^k)+(\Phi^k)^*\circ \psi^{k+1},\hspace{3mm}\emph{\text{for}}\hspace{1.7mm}k=0,...,H-1,\\
	\psi^{H}:&=T_eL_{\bar{g}^H}^*\circ\dD_g F(\bar{g}^H),
	\end{split}
	\end{equation}
	with the $\ell$ functions being determined by \eqref{leftell}. Then, the total cost of \eqref{optprob} satisfies
	\begin{align}\label{eq_lemma1}
	\dD_{u^k}|_{\bar{U}} J=\ell_u^k(\bar{g}^k,\bar{u}^k)+(\emph{\text B}^k)^{*}\circ\psi^{k+1}.
	\end{align}
\end{lemma}
\begin{proof}
	The proof is in appendix \ref{appproof1}.
\end{proof}

\begin{ass}
\label{ass1}
(i) The controls search space $\mathcal{U}\ni U$ is compact, (ii) $Q_{uu}^k$ remains positive definite for $k=0,...,H-1$.
\end{ass}	

The convergence properties of DDP on Lie groups are established by the following theorem and its corollary.
\begin{theorem}\label{th1}
Consider the discrete-time optimal control problem in \eqref{optprob}, with $J$ being the cost function. Let $\bar{U}$ be a nominal control sequence, and $\delta U_\star$ contain the control updates from \eqref{dustar}. Then, the following is true:
\[\mathsf D_U|_{\bar{U}}J\cdot\delta U=-\gamma\sum_{i=0}^{H-1}\langle Q_u^i,(Q_{uu}^i)^{-1}( Q_u^i)\rangle+O(\gamma^2).\]
\end{theorem}
\begin{proof}
	First, we prove the more generic result:
		\begin{equation}\label{eq0_lemma1}
		\begin{split}
		&\sum_{i=k}^{H-1}  \dD_{u^i}|_{\bar{U}} J\cdot\delta u^i =\\
		&-\gamma \sum_{i=k}^{H-1} \langle Q_u^i, (Q_{uu}^i)^{-1}(Q_u^i)\rangle+ \langle\mathcal{V}_g^k - \psi^k,\zeta^k\rangle+O(\gamma^2).
		\end{split}
		\end{equation}
Let us consider the case when $k\equiv H-1$:
\begin{equation*}
\begin{split}
\dD_{u^{H-1}}|_{\bar{U}} J&\overset{\eqref{eq_lemma1}}{=}\ell_u^{H-1}+(\text B^{H-1})^*\circ\psi^H\\
&\overset{\eqref{psi}}{=}\ell_u^{H-1}+(\text B^{H-1})^*\circ T_eL_{\bar{g}^H}^*\circ\dD_g F\\
&\overset{\eqref{vfinal}}{=}\ell_u^{H-1}+(\text B^{H-1})^*\circ\mathcal{V}_g^H\\
&\overset{\eqref{Qfunctions}}{=}Q_u^{H-1}.
\end{split}
\end{equation*}
Moreover, from \eqref{psi}
\begin{equation}
\label{qghmin1}
\psi^{H-1}=\ell_g^{H-1}+(\Phi^{H-1})^*\circ \mathcal{V}_g^H\overset{\eqref{Qfunctions}}{=}Q_g^{H-1}.
\end{equation}
It is thus clear from \eqref{dustar} that
\begin{equation*}
\begin{split}
&\dD_{u^{H-1}}|_{\bar{U}} J\cdot\delta u^{H-1}=\\
&-\gamma\langle Q_u^{H-1}, (Q_{uu}^{H-1})^{-1}(Q_u^{H-1})+(Q_{uu}^{H-1})^{-1}\circ Q_{ug}^{H-1}(\zeta^{H-1})\rangle\\
&\overset{\eqref{backv}}{=}-\gamma\langle Q_u^{H-1}, (Q_{uu}^{H-1})^{-1}(Q_u^{H-1})\rangle-\langle\mathcal{V}_g^{H-1}-Q_g^{H-1},\zeta^{H-1}\rangle\\
&\overset{\eqref{qghmin1}}{=}-\gamma\langle Q_u^{H-1}, (Q_{uu}^{H-1})^{-1}(Q_u^{H-1})\rangle-\langle\mathcal{V}_g^{H-1}-\psi^{H-1},\zeta^{H-1}\rangle.
\end{split}
\end{equation*}
Now, assume that identity \eqref{eq0_lemma1} is satisfied for $k+1$. Our goal is to show that it holds for $k$ as well. Note that by using a similar argument as in \cite[Lemma 4]{liao000}, one can show that $\zeta^k=O(\gamma)$ and $\delta u^k=O(\gamma)$, $\forall k$. Then, we obtain
\begin{equation*}
\begin{split}
&\sum_{i=k}^{H-1}  \dD_{u^i}|_{\bar{U}} J\cdot\delta u^i \overset{\eqref{eq0_lemma1}}{=} D_{u^k}|_{\bar{U}} J\cdot\delta u^k\\
&\hspace{1.3mm}-\gamma \sum_{i=k+1}^{H-1} \langle Q_u^i, (Q_{uu}^i)^{-1}(Q_u^i)\rangle+ \langle\mathcal{V}_g^{k+1} - \psi^{k+1},\zeta^{k+1}\rangle+O(\gamma^2)\\
&\stackrel{\eqref{linzeta1},\eqref{eq_lemma1}}{=\joinrel=\joinrel=\joinrel=}\langle\ell_u^k+(\text B^k)^*\circ\psi^{k+1},\delta u^k\rangle-\gamma \sum_{i=k+1}^{H-1}\langle Q_u^i, (Q_{uu}^i)^{-1}(Q_u^i)\rangle\\
&\quad+ \langle\mathcal{V}_g^{k+1} - \psi^{k+1},\Phi^{k}\zeta^{k}+\text B^{k}\delta u^{k}\rangle+O(\gamma^2)\\
&=\langle\underbrace{\ell_u^k+(\text B^k)^*\circ\mathcal{V}_g^{k+1}}_{Q_u^k},\delta u^k\rangle-\gamma \sum_{i=k+1}^{H-1} \langle Q_u^i, (Q_{uu}^i)^{-1}(Q_u^i)\rangle\\
&\quad+\langle(\Phi^{k})^*\circ\mathcal{V}_g^{k+1} - \underbrace{(\Phi^{k})^*\circ\psi^{k+1}}_{\psi^k-\ell_g^k},\zeta^{k}\rangle+O(\gamma^2)\\
&\overset{\eqref{dustar}}{=}-\gamma\sum_{i=k}^{H-1}\langle Q_u^i, (Q_{uu}^i)^{-1}(Q_u^i)\rangle+\langle\underbrace{(\Phi^{k})^*\circ\mathcal{V}_g^{k+1}+\ell_g^k}_{Q_g^k}-\psi^k,\zeta^k\rangle\\
&\quad+\langle \underbrace{-Q_{gu}^k\circ(Q_{uu}^k)^{-1}\circ Q_u^k}_{\mathcal{V}_g^k-Q_g^k},\zeta^k\rangle+O(\gamma^2)\\
&=-\gamma\sum_{i=k}^{H-1}\langle Q_u^i, (Q_{uu}^i)^{-1}(Q_u^i)\rangle+ \langle\mathcal{V}_g^k - \psi^k,\zeta^k\rangle+O(\gamma^2).
\end{split}
\end{equation*}

Finally, since $g^0$ is fixed, we have $\zeta^0=0$. Thus, evaluating eq. \eqref{eq0_lemma1} at $k\equiv0$, gives: $\dD_U|_{\bar{U}}J\cdot\delta U=\sum_{i=0}^{H-1}  \dD_{u^i}|_{\bar{U}} J\cdot\delta u^i=-\gamma\sum_{i=0}^{H-1}\langle Q_u^i, (Q_{uu}^i)^{-1}(Q_u^i)\rangle+O(\gamma^2)$, which concludes the proof.
\end{proof}

\begin{cor}\label{thconv2}
Suppose Assumption \ref{ass1} holds. Then, Algorithm \ref{alg:ddp} will converge to a stationary solution of \eqref{optprob}.
\end{cor}
\begin{proof}
The proof is in appendix \ref{appproof2}.
\end{proof}
\begin{remark}\label{rem:theta}
Notice that $T\mathbb{R}^m\simeq\mathbb{R}^m$ is simply the set of all column vectors in $\mathbb{R}^m$, while its dual comprises of row vectors in $\mathbb{R}^{1\times m}$. Hence, we account for Assumption \ref{ass1}-(ii) by setting
\begin{equation*}
Q_{uu}^k\leftarrow Q_{uu}^k+\lambda I_m,
\end{equation*}
at each $t_k$. This modification is incorporated in step \ref{stepQ} of Algorithm \ref{alg:ddp}. Here, $I_m\in\mathbb{R}^{m\times m}$ is the identity matrix, and $\lambda>0$ is a regularization parameter that enforces the positive definiteness of $Q_{uu}^k$ for all $k$ \cite{Nocedal}.
\end{remark}

{\it Convergence rate:} In \cite{ddp0} the authors showed that DDP achieves locally quadratic convergence rates forb Euclidean spaces. In our simulated examples we show that such behavior is possible also for Lie groups, when a second-order expansion scheme from section \ref{sec:zetaexp} is used. Developing the corresponding theoretical analysis is a topic under investigation.

\section{Simulations}\label{sec:sim}
The proposed scheme is employed here to control a mechanical system in simulation. We will find that under certain specifications, the Lie-theoretic formulation of section \ref{sec:ddp_der} can be implemented via simple matrix multiplications. Finally, numerical results are included to demonstrate the behavior and efficiency of our algorithm.

We consider the dynamics of a rigid satellite whose state evolves on the tangent bundle $TSO(3)$. The continuous equations of motion are given by \cite{crouchspacecraft}
\begin{equation}
\label{conteqs}
\begin{split}
\dot{R}&=R\widehat{\Omega},\\
\dot{\Omega}&=\mathbb{I}^{-1}\big((\mathbb{I}\Omega)\times\Omega+\mathbb{H}u\big).
\end{split}
\end{equation}
Here, $R\in SO(3)$ is the rotation matrix, $\Omega\in\mathbb{R}^3$ the body-fixed velocity, and $u\in\mathbb{R}^m$ the control input vector. Moreover, $\mathbb{I}\in\mathbb{R}^{3\times3}$ denotes the inertia tensor, and $\mathbb{H}\in\mathbb{R}^{3\times m}$ is a matrix whose columns represent the axis about which the control torques are applied. The isomorphism $\widehat{\cdot}:\mathbb{R}^3\rightarrow\mathfrak{so}(3)$ maps a column vector into a skew-symmetric matrix as follows:
\[
\begin{bmatrix}
\mathrm{x}_{1}\\\mathrm{x}_{2}\\\mathrm{x}_{3}
\end{bmatrix}^{\widehat{}}:=\begin{bmatrix}0&-\mathrm{x}_{3}&\mathrm{x}_{2}\\
\mathrm{x}_{3}&0&-\mathrm{x}_{1}\\
-\mathrm{x}_{2}&\mathrm{x}_{1}&0
\end{bmatrix}.
\]
Its inverse, $\vc{\cdot}:\mathfrak{so}(3)\rightarrow\mathbb{R}^3$, is defined such that $\vc{\widehat{\mathrm x}}=\rm x$, for all $\rm x\in\mathbb{R}^3$. Through these mappings we can treat the state as $g=(R,\Omega)\in SO(3)\times\mathbb{R}^3$ (see, e.g., \cite{sac} for more details).

Next, we discretize \eqref{conteqs}. Let us use the forward Euler method, which reads:
\begin{equation}
\label{disceqs}
\begin{split}
R^{k+1}&=R^{k}\exp_{SO(3)}{\big(\widehat{\Omega^{k}}\Delta t\big)},\\
\Omega^{k+1}&=\underbrace{\Omega^k+\Delta t\mathbb{I}^{-1}\big((\mathbb{I}\Omega^k)\times\Omega^k+\mathbb{H}u^k\big)}_{=:f_{\Omega}^k(R^k,\Omega^k,u^k)},
\end{split}
\end{equation}
with $\Delta t$ being the time-step, and $\exp_{SO(3)}(\cdot)$ being the usual matrix exponential. In what follows, let $I_n$, $0_{n\times m}$  denote the $n\times n$ identity and $n\times m$ zero matrix, respectively. Additionally, $||C||_{m,S}:=\sqrt{\text{trace}[C^\top SC]}$ is the weighted Frobenius matrix norm, and $||c||_{v,S}:=\sqrt{c^\top Sc}$ corresponds to the standard Euclidean weighted norm, for all $C,S\in\mathbb{R}^{n\times n}$, $c\in\mathbb{R}^{n}$ and $S>0$.

Our optimal control problem is formulated as:
\begin{equation}
\label{optprobR}
\begin{split}
&\min_{\{u^k\}_0^{H-1}}\hspace{1.7mm} \big[\sum_{k=0}^{H-1}\Lambda^k(R^k,\Omega^k,u^k)+F(R^H,\Omega^H)\big]\\
&\hspace{4mm}\text{s.t.}\quad \text{eq. } \eqref{disceqs},\quad R^0 = \bar{R}^0,\quad \Omega^0 = \bar{\Omega}^0,
\end{split}
\end{equation}
with
\begin{equation}
\label{costR}
\begin{split}
&\Lambda^k(R^k,\Omega^k,u^k):=\frac{1}{2}||u^k||^2_{v,S_u},\\ 
&F(R^H,\Omega^H):=\frac{1}{2}||I_3-R_d^\top R^H||_{m,S_R}^2+\frac{1}{2}||\Omega^H-\Omega_{d}||_{v,S_\Omega}^2.
\end{split}
\end{equation}
In this setting, the algorithm will penalize trajectories with high control inputs, and terminal state far from $(R_d,\Omega_d)$.

We now compute the derivatives of $\Lambda^k$ and $F$ which are used by the $Q$ functions of DDP. Consider, first, the terminal cost. One will have
\begin{flalign}
\label{gradF}
&\nonumber\langle\dD F(g^H),g^H\zeta\rangle=\left.\frac{\mathrm d}{\rm ds}\right|_{s=0}F(g^H\exp(s\zeta))=\\
&\nonumber\text{trace}\big[-(R_d^\top R^H\eta)^\top S_R(I_3-R_d^\top R^H)\big]+\rho^\top S_\Omega(\Omega^H-\Omega_{d})=\\
&\big(\vc{2\text{skew}(S_RR_d^\top R^H)}\big)^\top\vc{\eta}+\big(S_\Omega(\Omega^H-\Omega_{d})\big)^\top\rho,&
\end{flalign}
\begin{flalign}
\label{hessF}
&\nonumber\langle\text{Hess}^{(0)} F(g^H)(g^H\zeta_1),g^H\zeta_2\rangle=\\
&\nonumber\frac{1}{2}\big((g^H\zeta_1)(g^H\zeta_2)+(g^H\zeta_2)(g^H\zeta_1)\big)(F(g^H))=\\
&\nonumber \frac{1}{2}\scalebox{1.11}{$\left.\frac{\partial^2\big(F(g^H\exp(s_1\zeta_1)\exp(s_2\zeta_2))+F(g^H\exp(s_2\zeta_2)\exp(s_1\zeta_1))\big)}{\partial s_1\partial s_2}\right|_{0}$}=\\
&(\vc{\eta_1})^\top\big(\text{sym}(\text{trace}[S_RR_d^\top R^H]I_3-S_RR_d^\top R^H)\big)\vc{\eta_2}+\rho_1^\top S_\Omega\rho_2,&
\end{flalign}
where we have used the following identities: $\exp(\cdot)\equiv\exp_{TSO(3)}(\cdot)$, $\text{sym}(\mathsf N):=\frac{1}{2}(\mathsf N+\mathsf N^\top)$, $\text{skew}(\mathsf N):=\frac{1}{2}(\mathsf N-\mathsf N^\top)$, $\zeta_i=(\eta_i,\rho_i)\in\mathfrak{so}(3)\times\mathbb{R}^3$, $\text{trace}[\text{sym}(\mathsf N)\text{skew}(\mathsf M)]=0$, $\text{trace}[\widehat{\mathsf x}^\top\mathsf S]=2\mathsf x^\top\vc{\text{skew}(\mathsf S)}$, $\text{trace}[\widehat{\mathsf x}^\top\mathsf S\widehat{\mathsf y}]=\mathsf y^\top(\text{trace}[\mathsf S]I_3-\mathsf S)\mathsf x$, for all $\mathsf N, \mathsf M\in\mathbb{R}^{n\times n}$, $\mathsf x,\mathsf y\in\mathbb{R}^3$, $\mathsf S\in\mathbb{R}^{3\times3}$. Regarding the first and second equalities of \eqref{hessF}, one can refer to \cite[page 319]{Mahony2002} and \cite[eq. 2.12.5]{varadarajanbook}, respectively.

Now, let $\vd{\cdot}:\mathfrak{g}^*\rightarrow\mathbb{R}^{1\times n}$ be a mapping such that $\langle\mu,\mathsf x\rangle=\vd{\mu}\vc{\mathsf x}$, with $\vc{\mathsf x}\in\mathbb{R}^n$ being the (column) vector representation of $\mathsf x\in\mathfrak{g}$. In our example, this is equivalent to specifying the basis of $\mathfrak{so}^*(3)$ as $\{(\widehat{e_i})^\top\}_1^3$ (where $\{e_i\}_1^3$ denotes the canonical basis in $\mathbb{R}^3$), and using $\langle\phi,\zeta\rangle=\frac{1}{2}\text{trace}(\phi\zeta)$, for each $\phi\in\mathfrak{so}(3)^*$, $\zeta\in\mathfrak{so}(3)$. Let also $\md{\mathsf H}$ denote the matrix representation of a linear operator $\mathsf H:\mathfrak{g}\rightarrow\mathfrak{g}^*$, so that $\vd{\mathsf H(\zeta)}=(\vc{\zeta})^\top\md{\mathsf H}$. Then, equations \eqref{leftell}, \eqref{costR}, \eqref{gradF} and \eqref{hessF} imply
\begin{equation*}
\begin{split}
&\ell_u^k(g^k,u^k)=S_uu^k,\quad\ell_{uu}^k(g^k,u^k)=S_u,\quad\vd{\ell_g^k(g^k,u^k)}={0}_{1\times6},\\
&\md{\ell_{gg}^k(g^k,u^k)}={0}_{6\times6},\quad\md{\ell_{gu}^k(g^k,u^k)}={0}_{6\times m},\\
&\vd{T_eL_{g^H}^*\circ\mathsf DF(g^H)}=\begin{bmatrix}
2\vc{\text{skew}(S_RR_d^\top R^H)}\\S_\Omega(\Omega^H-\Omega_d)
\end{bmatrix}^\top,\\
&\md{T_eL_{g^H}^*\circ\text{Hess}^{(0)}F(g^H)\circ T_eL_{g}^H}=\\
&\hspace{18mm}\begin{bmatrix}
\text{sym}(\text{trace}[S_RR_d^\top R^H]I_3-S_RR_d^\top R^H)&{0}_{3\times3}\\{0}_{3\times3}&S_\Omega
\end{bmatrix}.
\end{split}
\end{equation*}
Notice that under these identifications, each step of Algorithm \ref{alg:ddp} can be implemented through matrix/vector products. For instance, $Q_{gg}$ from \eqref{Qfunctions}, and $\mathcal{V}_g$ from \eqref{backv} become
\begin{equation}
\label{Qmatrices}
\begin{split}
\md{Q_{gg}^k}=&\md{\ell_{gg}^k}+(\varPhi^k)^\top\md{\mathcal{V}_{gg}^{k+1}}\varPhi^k+\sum_{i=1}^6(\vd{\mathcal{V}_{g}^{k+1}})_i\varTheta_{(i)}^k,\\
\vd{\mathcal{V}_g^k}=&\vd{Q^k_g}-Q^k_u(Q_{uu}^k)^{-1}\md{Q^k_{gu}},
\end{split}
\end{equation}
where $\varPhi$, $\varTheta$ denote the matrix representations of $\Phi$, $\Theta$ respectively (see eq. \eqref{linzetamatrix}). The remaining terms can be computed similarly. We observe that the expressions in \eqref{Qmatrices} accord (up to the first order only) with those used in \cite{kobilarovddplie, liegroupddp_estimation}. Nonetheless, the derivation of section \ref{sec:ddp_der} holds for generic basis and pairing selections on $(\mathfrak{g}^*,\mathfrak{g})$. 

It remains to determine the linearization matrices for \eqref{disceqs}. We will apply the corresponding expressions from section \ref{sec:zetaexp} and appendix \ref{app:matrixform}, by setting $R\rightarrow\chi$ and $\Omega\rightarrow\vc{\xi}$. 
Direct differentiation of \eqref{disceqs} yields
\begin{equation*}
\begin{split}
&\vd{T_eL_{R^k}^*\circ\mathsf D_R(f_\Omega^k)_i} ={0}_{1\times3}, \quad \mathsf D_u(f_\Omega^k)=\Delta t\mathbb{I}^{-1}\mathbb{H},\\
&\mathsf D_\Omega(f_\Omega^k)=I_3+\Delta t\mathbb{I}^{-1}\big(-\widehat{\Omega}\mathbb{I}+\widehat{\mathbb{I}\Omega}\big),\\
&\dD_{\Omega_j}\dD_{\Omega_i}(f_\Omega^k)=\Delta t\mathbb{I}^{-1}\big((\widehat{\mathbb{I}e_i})e_j+(\widehat{\mathbb{I}e_j})e_i\big),\\
&\varTheta^k_{\chi\chi(i)}={0}_{6\times6},\quad \varTheta^k_{\chi\xi(i)}={0}_{6\times6},\quad\text{for all}\hspace{1.8mm}i>3,\\
&\Gamma_{(i)}^k={0}_{6\times m},\quad \Delta_{(i)}^k={0}_{m\times6},\quad \Xi_{(i)}^k={0}_{m\times m},\quad\text{for all}\hspace{1.8mm} i.
\end{split}
\end{equation*}
Lastly, note that in $SO(3)$, one has $\vc{\text{Ad}_R(\hat{\eta})}=R\eta$ and $\vc{\text{ad}_{\hat \zeta}(\hat{\eta})}=\hat{\zeta}\eta$.

{\it Numerical results:} Table \ref{tab:par} includes the parameters that were used in our simulations. The convergence criterion was set to $|J_{(i)}-J_{(i-1)}|\leq10^{-8}$, with $J_{(i)}$ being the cost at iteration $i$. When $Q_{uu}^k$ was found non positive definite for some $k$, the regularization parameter $\lambda$ from remark \ref{rem:theta} was increased as $\lambda\leftarrow1.9\lambda$, until satisfying the corresponding condition. In addition, $\mathsf h$ from step \ref{steph} of Algorithm \ref{alg:ddp} was set to $1/3$. Lastly, DDP was initialized with zero nominal controls. The obtained (sub)optimal state trajectory and control sequence are given in Figure \ref{fig:trajectories}, along with the desired final states, $R_d$ and $\Omega_d$. In these graphs, we have plotted the attitude by using a unit quaternion representation for each rotation matrix.

Next, we evaluate the effect the linearization schemes \eqref{linzeta} and \eqref{linzetamatrix} have on DDP's performance. In particular, we compare between using the linear terms only, as opposed to applying the full second-order expansion. We find that the first-order scheme does much better at the early stages of optimization, but fails to give superlinear convergence. In contrast, employing the second-order terms leads to $Q_{uu}^k$ being non positive definite in the first iterations, which slows down cost improvement. As we approach our solution, though, the higher order terms allow for quadratic-like convergence rates.

We propose switching between the two schemes depending on the relative change of the cost function. In this way, we can leverage the benefits of each approach and arrive faster at a stationary solution. Specifically, we begin to optimize by applying the linear terms only. When $J_{(i)}/J_{(0)}<\sigma$ (for some prespecified bound $\sigma$), the full expansion is used. In our example we picked $\sigma=0.1$, and the switching occurred after the first iteration of DDP. The aforementioned results and observations are illustrated in Figure \ref{fig:conv}.
\begin{table}
	\centering
	\caption{Parameter values used in simulations ($Rot_i(\theta)$ denotes an intrinsic rotation about the $i$ axis by an angle $\theta$).}
	\label{tab:par}
	\begin{tabular}{cc||cc}
		\toprule
		$t_f$ & 3 &  $\bar{R}^0$ &  $I_3$\\
		
		$\Delta t$ & $0.01$ &  $\bar{\Omega}^0$ &  $\textbf{0}_{3\times1}$\\
		
		$\mathbb{I}$ & $\text{diag}(10,11.1,13)$ & $R_d$  &  $Rot_x(30^\circ)Rot_z(70^\circ)$\\
		
		$\mathbb{H}$ & $I_3$ &  $\Omega_d$ &  $\textbf{0}_{3\times1}$\\
		$S_u$ & $0.1I_3$ & $S_R,S_\Omega$ & $10^4I_3$ \\
		\bottomrule
	\end{tabular}
\end{table}
\begin{figure*}[!t]
	\centering
	\includegraphics[width=0.316\textwidth]{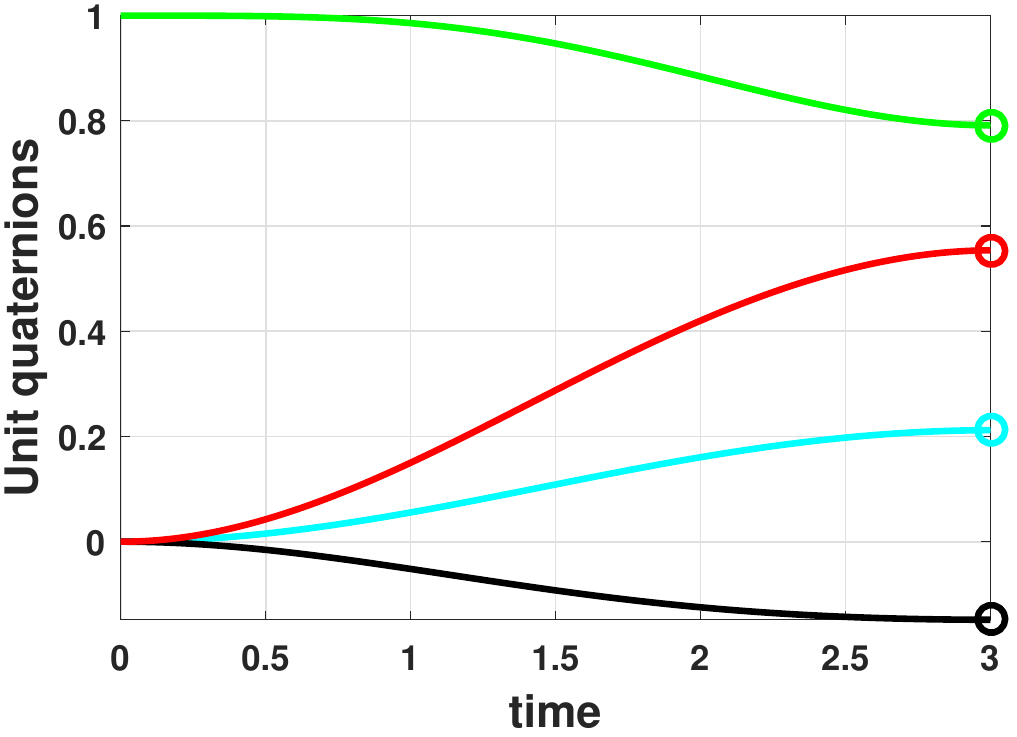}
	\includegraphics[width=0.32\textwidth]{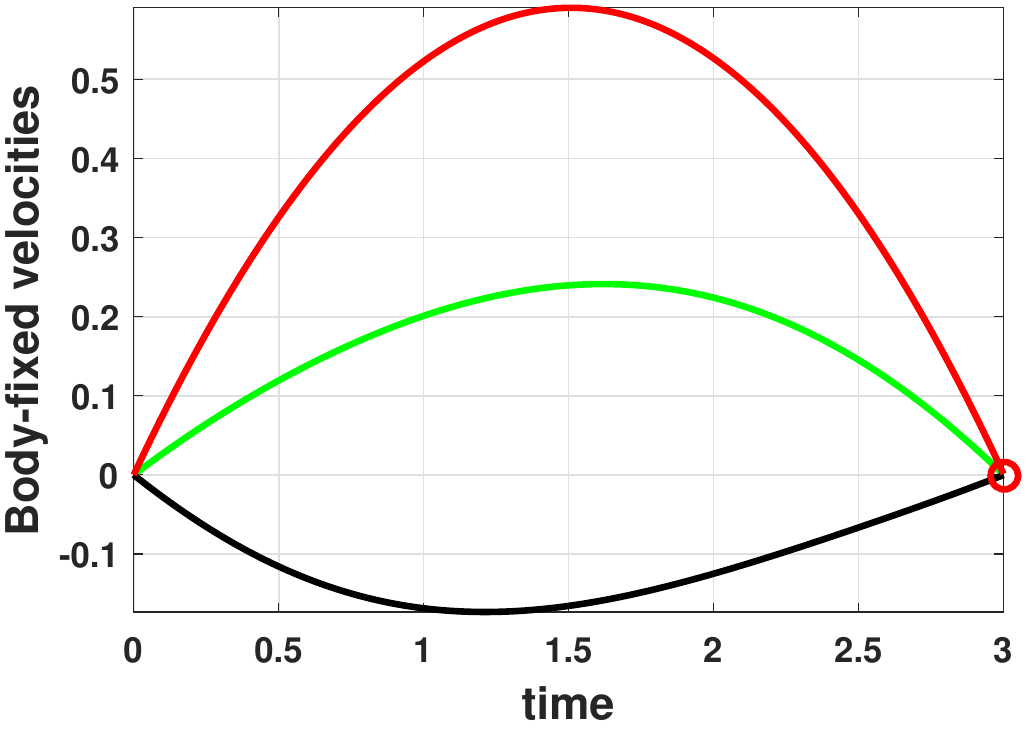}
	\includegraphics[width=0.31\textwidth]{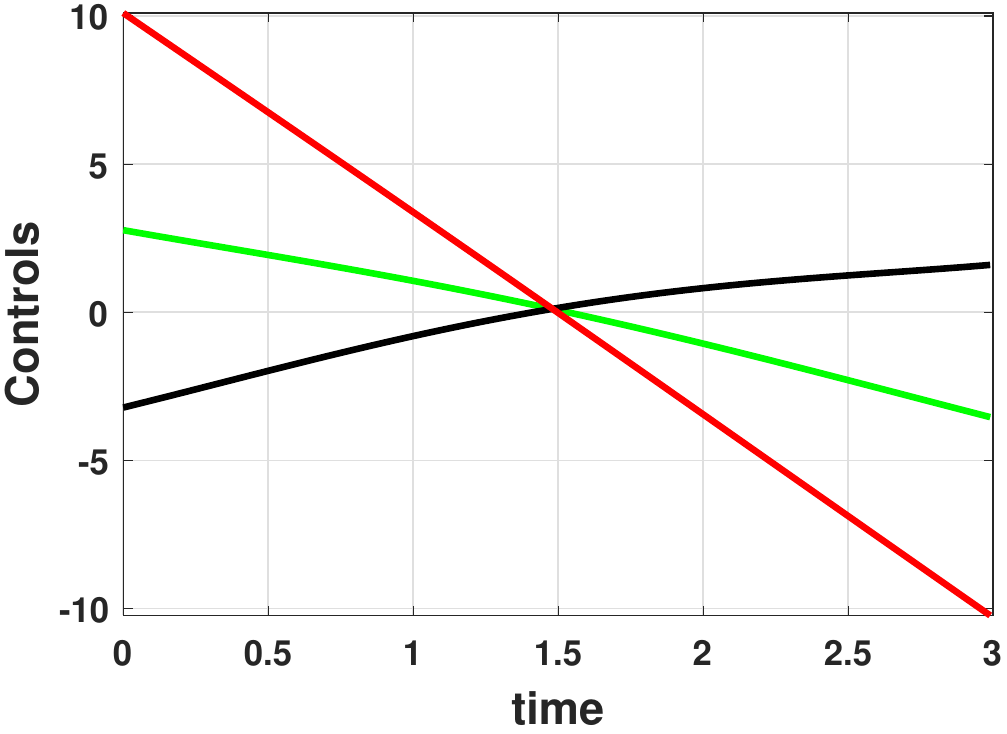}
	
	\caption{Illustration of DDP's (sub)optimal solution - the obtained state trajectory and controls are depicted. The first graph from the left uses unit quaternions to represent the sequence of rotation matrices over time. Each circle at the terminal time instant corresponds to the desired states, $R_d$ or $\Omega_d$.}
	\label{fig:trajectories}
\end{figure*}
\begin{figure}
	\centering
	\subfigure{
		\includegraphics[width=0.47\textwidth]{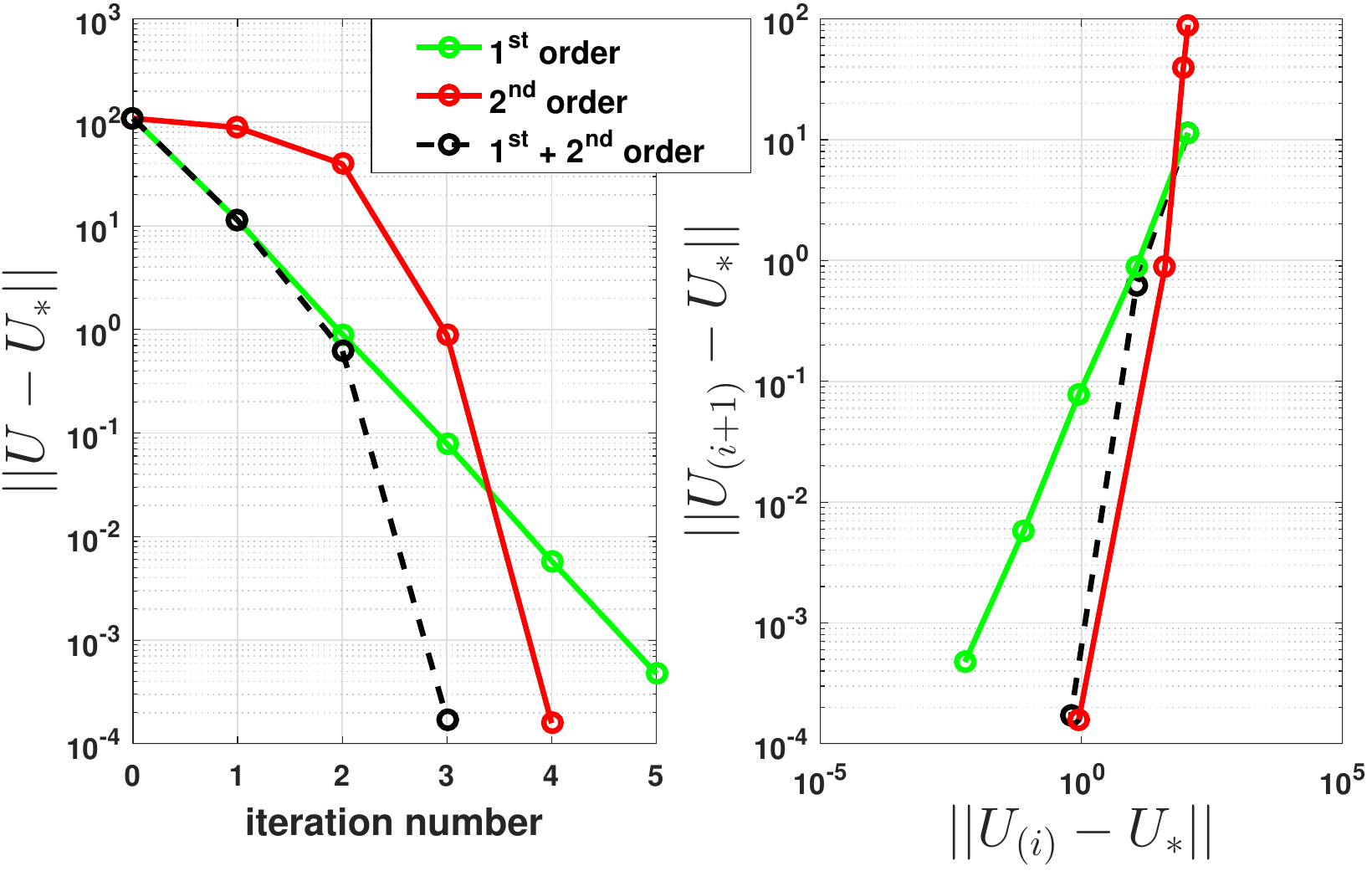}}
	\caption{Convergence-related results of DDP for different linearization schemes. The solid green line is associated with a linear expansion, and the solid red line with a quadratic expansion. $U_*$ corresponds to the (sub)optimal solution each setting converges to, and $U_{(i)}$ denotes the controls sequence at iteration $i$. The second-order scheme achieves locally superlinear convergence rates, but is slower in the first iterations. In light of this, the dashed black line employs higher-order terms from the second iteration, obtaining, thus, the solution in fewer steps.}\label{fig:conv}
\end{figure}

{\it Comparison with off-the-self optimization method:} A common approach for solving discrete optimal control problems is using off-the-self optimization solvers. In this way, existing algorithms can be applied with little to no modification. Moreover, feasible trajectories can be generated by simply treating the dynamics as equality constraints. In the context of geometric control, \cite{kobilarovdiscrete} applied a direct-optimization method for controlling a variational integration scheme of mechanical systems.

As a benchmark comparison, we employed $\textsc{Matlab}$'s built-in SQP implementation for solving problem \eqref{optprobR}. Regarding the decision variables vector, we simply selected $(\text{vec}(R^0)^\top,...,\text{vec}(R^H)^\top,(\Omega^0)^\top,...,(\Omega^H)^\top,U^\top)^\top\in\mathbb{R}^{12H+m(H-1)}$. We also provided derivative information for the cost function and equality constraints to speed up convergence. Lastly, the feasibility tolerance for the dynamics was set to $10^{-6}$.

Figure \ref{fig:sqp} compares the total cost per iteration between Differential Dynamic Programming and SQP. The two methods reach the same solution, but DDP requires much fewer iterations. In addition, our $\textsc{Matlab}$ implementation of DDP converged in 1.9 s, with SQP being approximately 300 times slower. Last but not least, the SQP solver did not yield feasible dynamics until the $9^{\text{th}}$ iteration. This significant difference in performance is observed because a direct optimization method will (i) increase the dimension of the decision vector, (ii) search in a space under many equality constraints, and (iii) scale cubically with the time horizon.
\begin{figure}
	\centering
	\includegraphics[width=0.34\textwidth]{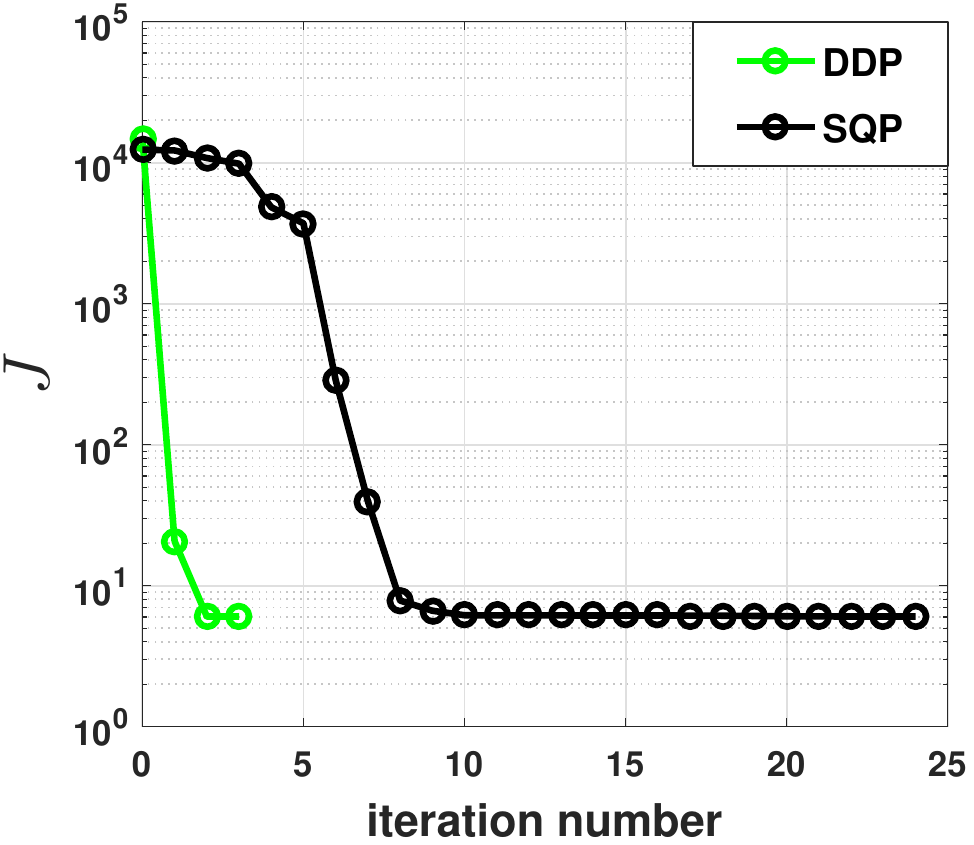}
	\caption{Comparison between Differential Dynamic Programming and SQP. The two approaches give the same solution, but DDP takes significantly less time and iterations. Furthermore, the SQP-solver yields infeasible dynamics for a large number of steps.}\label{fig:sqp}
\end{figure}

\section{Conclusion}\label{sec:conc}
In this paper we extended Differential Dynamic Programming from Euclidean models, to systems evolving on Lie groups. We focused on multiple aspects including the derivation, convergence properties, and practical implementation of the methodology. By utilizing a differential geometric approach, we handled problems defined on nonlinear configuration spaces. We also developed quadratic expansions for discrete mechanical systems, which were incorporated in our framework. The obtained geometric control algorithm preserved important characteristics of the original scheme, and, thus, outperformed standard optimization methods in simulation.

Despite the aforementioned contributions, certain topics require further investigation. For example, developing an analysis for the convergence rate of geometric Differential Dynamic Programming is necessary to establish the properties of the algorithm. Moreover, our work relied on the Cartan-Schouten connections. For a different affine connection, the form of Algorithm \ref{alg:ddp} is expected to vary. Finally, an interesting extension would be to develop a stochastic version of DDP on Lie groups. This could further increase its applicability when dealing with real autonomous systems.

\appendices

\section{On the Hessian operators of Cartan connections} \label{app:hess}
This section is a review of \cite{Mahony2002} with the simple addition of including all three Cartan connections in the analysis.

Let $gx$, $gy\in\mathfrak{X}$ be two left-invariant vector fields, with $x,y\in\mathfrak{g}$ and $g\in G$. As shown in \cite[page 319]{Mahony2002}, the Hessian of a function $\mathsf f:G\rightarrow\mathbb{R}$ with respect to the (0) connection satisfies
\[\text{Hess}^{(0)}\mathsf f(g)(gx)(gy)=\frac{1}{2}\big((gx)(gy)+(gy)(gx)\big)(\mathsf f(g)).\]

Now, from the definition of the torsion tensor and the corresponding connection functions (see section \ref{sec:prel}), it is easy to show that: $\mathcal T^{(0)}(gx,gy)= 0$, $\mathcal T^{(-)}(gx,gy)=-[gx,gy]$, and $\mathcal T^{(+)}(gx,gy)=[gx,gy]$. Hence, one can obtain for the remaining Hessians:
\begin{equation*}
\begin{split}
&\text{Hess}^{(-)}\mathsf f(g)(gx)(gy)=(gx)\big((gy)(\mathsf f(g))\big)=\\
&\big[\frac{1}{2}\big((gx)(gy)+(gy)(gx)\big)+\frac{1}{2}\big((gx)(gy)-(gy)(gx)\big)\big](\mathsf f(g))=\\
&\text{Hess}^{(0)}\mathsf f(g)(gx)(gy)+\frac{1}{2}\mathcal T^{(-)}(gy,gx)(\mathsf f(g)),
\end{split}
\end{equation*}
\begin{equation*}
\begin{split}
&\text{Hess}^{(+)}\mathsf f(g)(gx)(gy)=(gx)\big((gy)(\mathsf f(g))\big)-g[x,y](\mathsf f(g))=\\
&\big[\frac{1}{2}\big((gx)(gy)+(gy)(gx)\big)+\frac{1}{2}\big((gy)(gx)-(gx)(gy)\big)\big](\mathsf f(g))=\\
&\text{Hess}^{(0)}\mathsf f(g)(gx)(gy)+\frac{1}{2}\mathcal T^{(+)}(gy,gx)(\mathsf f(g)),
\end{split}
\end{equation*}
where we have used that $g[x,y]=[gx,gy]$ (see \cite{Gallier_notes} for proof). Since $T^{(+)}(gx,gx)=T^{(-)}(gx,gx)= 0$, $\forall x\in\mathfrak{g}$, we have that: $\text{Hess}^{(+)}\mathsf f(g)(gx)(gx)=\text{Hess}^{(-)}\mathsf f(g)(gx)(gx)=\text{Hess}^{(0)}\mathsf f(g)(gx)(gx)=(\text{Hess}^{(0)}\mathsf f(g))^*(gx)(gx)$. Therefore, when considering quadratic expansions as in \eqref{expandv}, we are only left with a symmetric second-order term.

\section{The Baker-Campbell-Hausdorff formula} \label{sec:bch}
The Baker-Campbell-Hausdorff (BCH) formula is stated in the following theorem. A more detailed treatment can be found in \cite{Gallier_notes} and \cite{varadarajanbook}.
\begin{theorem}
Let $G$ denote a Lie group, with $\mathfrak{g}$ being its Lie algebra. Then, there exists an open set $\mathfrak{g}_e^2\subseteq\mathfrak{g}\times\mathfrak{g}$ which contains $(0,0)$, such that for all $(\mathsf x,\mathsf y)\in\mathfrak{g}_e^2$
\[\exp(\mathsf x)\exp(\mathsf y)=\exp(\mu(\mathsf x,\mathsf y)).\]
Specifically, $\mu:\mathfrak{g}_e^2\rightarrow\mathfrak{g}$ is a real-analytic mapping that can be expanded as
\begin{equation}
\begin{split}
\label{bch}
\mu(\mathsf x,\mathsf y)=&\mathsf x+\mathsf y+\frac{1}{2}[\mathsf x,\mathsf y]+\frac{1}{12}\big([\mathsf x,[\mathsf x,\mathsf y]]+ [\mathsf y,[\mathsf y,\mathsf x]]\big)+\\
&\frac{1}{24}[\mathsf x,[\mathsf y,[\mathsf y,\mathsf x]]]+\text{higher order terms}
\end{split}
\end{equation}
\end{theorem}
\begin{proof}
See \cite[section 2.15]{varadarajanbook}.
\end{proof}
A closed-form version of \eqref{bch} that is linear with respect to each argument is given by
\begin{equation}
\label{bchlinear}
\begin{split}
\mu(\mathsf x,\mathsf y)&=\mathsf x+\text{dexp}^{-1}_{-\mathsf x}(\mathsf y)+O(||\mathsf y||^2)\\
&=\mathsf y+\text{dexp}^{-1}_{\mathsf y}(\mathsf x)+O(||\mathsf x||^2).
\end{split}
\end{equation}
Moreover, one can show the following \cite{varadarajanbook}
\begin{equation}
\label{bch1}
\exp(-\mathsf x)\exp(\mathsf x+\mathsf y)=\exp\big(\text{dexp}_{-\mathsf x}(\mathsf y)+O(||\mathsf y||^2)\big).
\end{equation}
Finally, note that the right-trivialized tangent of $\exp(\cdot)$ and its inverse are determined respectively by \cite{haier, iserles}
\begin{equation}
\label{dexp}
\text{dexp}_\eta(\zeta)=\sum_{j=0}^\infty\frac{1}{(j+1)!}\text{ad}^j_\eta\zeta,\quad\text{dexp}^{-1}_\eta(\zeta)=\sum_{j=0}^\infty\frac{B_j}{j!}\text{ad}^j_\eta\zeta,
\end{equation}
with $B_j$ being the {\it Bernoulli numbers} (i.e., $B_0=1$, $B_1=-\frac{1}{2}$, $B_2=\frac{1}{6}$, etc).

\section{First-order linearization of equation \eqref{etaexpression} - alternative proof} \label{order1eta}
Let us restate equation \eqref{etaexpression} below for convenience.
\begin{equation}
\label{etaexpression1}
\eta^{k+1}=\underbrace{\log\big(\exp(-\Delta t\bar{\xi}^k)\exp(\eta^k)\exp(\Delta t\bar{\xi}^k+\Delta t\delta\xi^k)\big)}_{=:RHS(\eta^k,\delta\xi^k)}.
\end{equation}
We will show that a first-order expansion can be obtained by using Taylor's formula directly. As expected, the result will match the linear terms of \eqref{etanew}.

Since $\eta^{k+1},\eta^k,\delta\xi^k\in\mathfrak{g}$ are all tangent vectors, the right-hand side of \eqref{etaexpression1} can be viewed as a mapping between vector spaces. Hence, we can expand about $\eta^k\equiv0$ and $\delta\xi^k\equiv0$ as follows
\begin{equation*}
\begin{split}
\eta^{k+1}=&RHS(0,0)+\mathsf D_\eta|_{(0,0)}RHS(\eta^k,\delta\xi^k)(\eta^k)+\\
&\mathsf D_{\delta\xi}|_{(0,0)}RHS(\eta^k,\delta\xi^k)(\delta\xi^k)+O(||(\eta^k,\delta\xi^k)||^2),
\end{split}
\end{equation*}
where
\begin{flalign*}
&RHS(0,0)=\log\big(\exp(-\Delta t\bar{\xi}^k)\exp(\Delta t\bar{\xi}^k)\big)=0,&
\end{flalign*}
\begin{flalign*}
&\mathsf D_\eta|_{(0,0)}RHS(\eta^k,\delta\xi^k)(\eta^k)=\\
&\mathsf D\log(e)\cdot\big(\exp(-\Delta t\bar{\xi}^k)(\mathsf D\exp(0)\cdot\eta^k)\exp(\Delta t\bar{\xi}^k)\big)=\\
&\text{Ad}_{\exp(-\Delta t\bar{\xi}^k)}(\eta^k),&
\end{flalign*}
\begin{flalign*}
&\mathsf D_{\delta\xi}|_{(0,0)}RHS(\eta^k,\delta\xi^k)(\delta\xi^k)=\\
&\mathsf D\log(e)\cdot\big(\exp(-\Delta t\bar{\xi}^k)\mathsf D\exp(\Delta t\bar{\xi}^k)\cdot(\Delta t\delta\xi^k)\big)=\\
&\Delta t\exp(-\Delta t\bar{\xi}^k)\text{dexp}_{(\Delta t\bar{\xi}^k)}(\delta\xi^k)\exp(\Delta t\bar{\xi}^k)=\\
&\Delta t\Ad_{\exp(-\Delta t\bar{\xi}^k)}\dexp_{(\Delta t\bar{\xi}^k)}(\delta\xi^k)=\\
&\Delta t\text{dexp}_{(-\Delta t\bar{\xi}^k)}(\delta\xi^k).&
\end{flalign*}
For the above expressions, we have used the chain rule and the following identities: $\exp(0)=e$, $\log(e)=0$, $\exp(-\rho)\exp(\rho)=e$, $\mathsf D\exp(0)\cdot\rho=\rho$, $\mathsf D\log(e)\cdot\rho=\rho$, and $\text{Ad}_{\exp(-\rho)}\text{dexp}_\rho(\lambda)=\text{dexp}_{(-\rho)}(\lambda)$.

A second-order expansion is also possible with this approach. For example, one could make use of higher-order covariant derivatives on mappings between manifolds. This in turn is associated with the selected affine connection (see discussion in \cite{lieprojection}). Due to the structure of \eqref{etaexpression1}, we prefer utilizing the BCH formula towards obtaining the quadratic terms. This yields a connection-independent scheme that relies on simple group operations.

\section{Matrix/vector form of linearized state perturbations} \label{app:matrixform}

Our goal is to obtain a matrix/vector form for equation \eqref{linzeta}. In what follows, $\{E_i\}_1^n$ denotes a basis for the Lie algebra $\mathfrak{g}'\ni\eta^k,\xi^k$. Let us also define the linear isomorphism $\mathsf{v}:\mathfrak{g}'\rightarrow\mathbb{R}^{n}$, such that $\vc{\mathsf x}=\vc{\sum_{i=1}^n{E}_i\mathsf x_i}:=(\mathsf x_1,...,\mathsf x_n)^\top$.

First, from the bilinearity of the adjoint representation, one has $\text{ad}_{\mathsf x}{\mathsf y}=\text{ad}_{(\sum_i\mathsf x_iE_i)}(\sum_j\mathsf y_jE_j)=\sum_i\sum_j\mathsf x_i\mathsf y_j\text{ad}_{E_i}E_j$. Thus, \eqref{etanew} can be rewritten as
\begin{equation}
\label{etamatrix}
\begin{split}
{\eta^{k+1}}\approx&\sum_i({\text{Ad}_{\exp(-\Delta t\bar{\xi}^k)}}E_i){\eta^k_i}+\sum_i(\Delta t{\text{dexp}_{(-\Delta t\bar{\xi}^k)}}E_i){\delta\xi^k_i}+\\
&\sum_{i,j}\big(\varOmega_{ij}\eta^k_i\eta^k_j+K_{ij}\eta^k_i\delta\xi^k_j\big),
\end{split}
\end{equation}
where we have defined:
\begin{equation}
\label{etamatrixterms}
\begin{split}
\varOmega_{ij}:=&\frac{\Delta t}{12}\big({\text{dexp}_{(-\Delta t\bar{\xi}^k)}}{\text{ad}_{E_i}}{\ad_{E_j}}{\bar{\xi}^k}+\\
&\frac{\Delta t}{2}{\text{ad}_{E_i}}{\text{ad}^2_{\bar{\xi}^k}}{E_j}+\frac{\Delta t}{2}{\ad_{E_i}}{\ad_{E_j}}{\bar{\xi}^k}-\\
&{\ad_{\dexp^{-1}_{(\Delta t\bar{\xi}^k)}{E_i}}}{\ad_{\dexp^{-1}_{(\Delta t\bar{\xi}^k)}{E_j}}}{\bar{\xi}^k}\big),\\
K_{ij}:=&\frac{\Delta t}{2}\big({\text{dexp}_{(-\Delta t\bar{\xi}^k)}}{\text{ad}_{E_i}}{E_j}+\\
&\frac{\Delta t}{6}\big({\ad_{\bar{\xi}^k}}{\ad_{E_j}}{E_i}+2{\ad_{E_j}}{\ad_{\bar{\xi}^k}}{E_i}+\\
&{\ad_{E_i}}{\ad_{\bar{\xi}^k}}{E_j}\big)\big).
\end{split}
\end{equation}
From eqs. \eqref{expandxi} and \eqref{etamatrix}, one can compute matrices $\varPhi^k\in\mathbb{R}^{2{n}\times2{n}}$, $B^k\in\mathbb{R}^{2n\times m}$, $\varTheta^k_{(i)}\in\mathbb{R}^{2{n}\times2{n}}$, $\varGamma^k_{(i)}\in\mathbb{R}^{2{n}\times m}$, $\varDelta^k_{(i)}\in\mathbb{R}^{m\times2{n}}$, and $\varXi^k_{(i)}\in\mathbb{R}^{m\times m}$, for $i=1,...,2{n}$, such that

\begin{equation}
\label{linzetamatrix}
\begin{split}
\vc{\zeta^{k+1}&}\approx\varPhi^k\vc{\zeta^k}+B^k\delta u^k+\\
\frac{1}{2}&\left(\begin{bmatrix}(\vc{\zeta^k})^\top\varTheta_{(1)}^k\vc{\zeta^k}\\\vdots\\(\vc{\zeta^k})^\top\varTheta_{(2{n})}^k\vc{\zeta^k}\end{bmatrix}\right.+\begin{bmatrix}(\vc{\zeta^k})^\top\varGamma^k_{(1)}\delta u^k\\\vdots\\(\vc{\zeta^k})^\top\varGamma^k_{(2{n})}\delta u^k\end{bmatrix}+\\
&\hspace{3mm}\begin{bmatrix}(\delta u^k)^\top\varDelta^k_{(1)}\vc{\zeta^k}\\\vdots\\(\delta u^k)^\top\varDelta^k_{(2{n})}\vc{\zeta^k}\end{bmatrix}+\left.\begin{bmatrix}(\delta u^k)^\top\varXi_{(1)}^k\delta u^k\\\vdots\\(\delta u^k)^\top\varXi_{(2{n})}^k\delta u^k\end{bmatrix}\right),
\end{split}
\end{equation}
with $\vc{\zeta^k}=((\vc{\eta^k})^\top,(\vc{\delta\xi^k})^\top)^\top$. Specifically, let us partition the transition matrices of \eqref{linzetamatrix} as follows
\begin{equation*}
\begin{split}
&\varPhi^k:=\begin{bmatrix}\varPhi^k_{\chi\chi}&\varPhi^k_{\chi\xi}\\\varPhi^k_{\xi\chi}&\varPhi^k_{\xi\xi}\end{bmatrix},\quad B^k:=\begin{bmatrix}B^k_{\chi}\\B^k_{\xi}\end{bmatrix},\\
&\varTheta_{(i)}^k:=\begin{bmatrix}\varTheta^k_{\chi\chi(i)}&\varTheta^k_{\chi\xi(i)}\\(\varTheta^k_{\chi\xi(i)})^\top&\varTheta^k_{\xi\xi(i)}\end{bmatrix},\quad\varGamma_{(i)}^k:=\begin{bmatrix}\varGamma^k_{\chi(i)}\\\varGamma^k_{\xi(i)}\end{bmatrix},
\end{split}
\end{equation*}
Then, it is easy to see that
\begin{equation*}
\begin{split}
&(\varPhi^k_{\chi\chi})_{ij}:=\vc{\text{Ad}_{\exp(-\Delta t\bar{\xi}^k)}E_j}_i,\hspace{1.3mm}(\varPhi^k_{\chi\xi})_{ij}:=\Delta t\vc{\text{dexp}_{(-\Delta t\bar{\xi}^k)}E_j}_i,\\
&(\varPhi^k_{\xi\chi})_{ij}:=\mathsf D_\chi(f_\xi^k)_i(\bar{g}^k,\bar{u}^k)(\bar{\chi}^kE_j),\\
&(\varPhi^k_{\xi\xi})_{ij}:=\mathsf D_\xi(f_\xi^k)_i(\bar{g}^k,\bar{u}^k)(E_j),\\
&B^k_{\chi}:={0},\quad B^k_{\xi}:=\mathsf D_u(f_\xi^k)(\bar{g}^k,\bar{u}^k),
\end{split}
\end{equation*}
where $(\#)_{ij}$ denotes here the $ij^\text{th}$ element of a matrix $``\#"$. Furthermore,
\begin{itemize}\item for each $1\leq i\leq n$:\end{itemize}
\begin{equation*}
\begin{split}
&(\varTheta^k_{\chi\chi(i)})_{jl}:=\begin{cases} 2\vc{\Omega_{jl}}_i, & j=l \\ \vc{\Omega_{jl}}_i+\vc{\Omega_{lj}}_i, & j\neq l \end{cases},\\
&(\varTheta^k_{\chi\xi(i)})_{jl}:=\vc{K_{jl}}_i,\quad(\varTheta^k_{\xi\xi(i)})_{jl}:=0,
\end{split}
\end{equation*}
\begin{equation*}
\varGamma^k_{(i)}:={0},\quad\varDelta^k_{(i)}:={0},\quad\varXi^k_{(i)}:={0},
\end{equation*}
\begin{itemize}\item while for each $i> n$:\end{itemize}
\begin{equation*}
\begin{split}
&(\varTheta^k_{\chi\chi(i)})_{jl}:=\text{Hess}^{(0)}_\chi(f_\xi^k)_i(\bar{g}^k,\bar{u}^k)(\bar{\chi}^kE_j)(\bar{\chi}^kE_l),\\
&(\varTheta^k_{\chi\xi(i)})_{jl}:=\mathsf D_\xi D_\chi(f_\xi^k)_i(\bar{g}^k,\bar{u}^k)(E_l)(\bar{\chi}^kE_j),\\
&(\varTheta^k_{\xi\xi(i)})_{jl}:=\mathsf D^2_\xi(f_\xi^k)_i(\bar{g}^k,\bar{u}^k)(E_j)(E_l),
\end{split}
\end{equation*}
\begin{equation*}
\begin{split}
&(\varGamma^k_{\chi(i)})_{jl}:=\mathsf D_{u_l} D_\chi(f_\xi^k)_i(\bar{g}^k,\bar{u}^k)(\bar{\chi}^kE_j),
\end{split}
\end{equation*}
\begin{equation*}
\varDelta^k_{(i)}:=(\varGamma^k_{(i)})^\top,\quad\varXi^k_{(i)}:=\mathsf D^2_u(f_\xi^k)_i(\bar{g}^k,\bar{u}^k).
\end{equation*}

\section{Proof of Lemma \ref{lemm:psi}}\label{appproof1}
Let  $\{\bar{u}^i+\upsilon^i\}_0^{H-1}$ denote small enough control perturbations about $\bar{U}$. Define also the concatenated vector $\mathsf U:=((\upsilon^0)^\top,...,(\upsilon^{H-1})^\top)^\top\in\mathbb{R}^{m(H-1)}$. Then, it is not hard to see that
 \begin{equation}
 \label{chainrule}
 \begin{split}
\mathsf D_U|_{\bar{U}}J(\{g^i\},\{u^i\})&=\mathsf D_{\mathsf U}|_{ 0}J(\{{g}^i_\epsilon\},\{{u}^i_\epsilon\})\\
&=\mathsf D_{\mathsf U}|_{ 0}J(\{\bar{g}^i\exp(\zeta^i)\},\{\bar{u}^i+\upsilon^i\}).
\end{split}
\end{equation}
$\{g^i_\epsilon\}$ denotes here the state trajectory under the perturbed controls $\{u_\epsilon^i\}=\{\bar{u}^i+\upsilon^i\}$ (i.e., $g_\epsilon^{i+1}=f^i(g^i_\epsilon,u^i_\epsilon)$), which can be written using exponential coordinates. Notice that the last term in \eqref{chainrule} is only a function of $\{\zeta^i\}$, $\{\upsilon^i\}$. Furthermore, $\{\upsilon^i\}_0^{H-1}\equiv\{0\}_0^{H-1}$ implies $\{\zeta^i\}_0^H\equiv\{0\}_0^H$. Now using the chain rule for mappings on manifolds yields
\begin{equation*}
\begin{split}
&\dD_{u^k}|_{\bar{U}} J\overset{\eqref{chainrule}}{=}\dD_{\upsilon^k}|_{ 0}J(\{\bar{g}^i\exp(\zeta^i)\},\{\bar{u}^i+\upsilon^i\})\\
&\overset{\eqref{J}}{=}\dD_{\upsilon^k}|_{ 0}\big[\sum_{i=k}^{H-1}\Lambda^i(\bar{g}^i\exp(\zeta^i),\bar{u}^i+\upsilon^i))+F(\bar{g}^H\exp(\zeta^H))\big]\\
&=\dD_{u^k}\Lambda^k(\bar{g}^k,\bar{u}^k)\\
&\quad+\dD_{g^{k+1}}\Lambda^{k+1}(\bar{g}^{k+1},\bar{u}^{k+1})\circ T_eL_{\bar{g}^{k+1}}\circ\dD\exp(0)\circ\left.\frac{\partial\zeta^{k+1}}{\partial\upsilon^k}\right|_{ 0}\\
&\quad+\cdots+\dD_{g^H} F(\bar{g}^H)\circ T_eL_{\bar{g}^H}\circ\dD\exp(0)\circ\\
&\hspace{18mm}\frac{\partial\zeta^{H}}{\partial\zeta^{H-1}}\circ\frac{\partial\zeta^{H-1}}{\partial\zeta^{H-2}}\circ\cdots\circ\left.\frac{\partial\zeta^{k+1}}{\partial\upsilon^k}\right|_{ 0}.
\end{split}
\end{equation*}
Finally, from the identity $\dD\exp(0)\cdot\rho=\rho$ and equations \eqref{leftell}, \eqref{linzeta1}, the above expression becomes
\begin{equation*}
\begin{split}
\dD_{u^k}|_{\bar{U}} J=&\ell_u^k+\dD_{g^{k+1}}\Lambda^{k+1}\circ T_eL_{\bar{g}^{k+1}}\circ\text{B}^k+\\
&\dD_{g^{k+2}}\Lambda^{k+2}\circ T_eL_{\bar{g}^{k+2}}\circ\Phi^{k+1}\circ\text{B}^k+\cdots\\
&+\dD_{g^H} F\circ T_eL_{\bar{g}^H}\circ\Phi^{H-1}\circ\Phi^{H-2}\circ\cdots\Phi^{k+1}\circ\text{B}^k\\
\overset{\eqref{psi}}{=}&\ell_u^k+(\text B^k)^{*}\circ\bigg[\ell_g^{k+1}+(\Phi^{k+1})^*\circ\big[\ell_g^{k+2}+(\Phi^{k+2})^*\circ\big[\cdots\\
&+(\Phi^{H-2})^*\big[\underbrace{\ell_g^{H-1}+(\Phi^{H-1})^*\circ \underbrace{T_eL_{\bar{g}^H}^*\circ\dD_{g^H} F}_{\psi^H}}_{\psi^{H-1}}\big]\cdots\big]\bigg]\\
\overset{\eqref{psi}}{=}&\cdots\\
\overset{\eqref{psi}}{=}&\ell_u^k+(\text B^k)^{*}\circ\psi^{k+1},
\end{split}
\end{equation*}
where we have omitted showing the explicit dependence on $\{\bar{g}^i\}$ and $\{\bar{u}^i\}$ for brevity.

\section{Proof of Corollary \ref{thconv2}}\label{appproof2}
Let $J_{(i)}$, $U_{(i)}$ denote respectively the total cost and control sequence  at iteration $i$. Define also $\Delta J_{(i)}:= J_{(i)}- J_{(i-1)}$. From Assumption \ref{ass1} and Theorem \ref{th1}, there exists $\gamma\in(0,1]$ small enough, such that $\Delta J_{(i)}<0$, for all $i$. Therefore, $J$ is monotonically decreasing and, thus\footnote{$J$ is assumed to be differentiable and, hence, continuous.}, there exists $U_*\in\mathcal{U}$ such that: $\lim_{i\rightarrow\infty}\Delta J_{(i)}=0$, with $\lim_{i\rightarrow\infty} J(U_{(i)})=J(U_*)$.

Now, by combining $\Delta J\rightarrow0$ and Theorem \ref{th1}, one has: $\langle Q_u^k, (Q_{uu}^k)^{-1}(Q_u^k)\rangle\rightarrow0\xRightarrow{Q_{uu}^k \text{ is p.d.}}Q_u^k\rightarrow 0$, for each $k$. Recalling that $\zeta^0=0$, we obtain from the control update \eqref{dustar}: $\delta u^0_\star=-(Q_{uu}^0)^{-1}(Q_u^0)=0$. From eq. \eqref{linzeta}, this in turn yields $\zeta^1=0$. Next, assume that $\delta u^k_\star=0$ and $\zeta^{k+1}=0$, for some $k>0$. Then, $\delta u^{k+1}_\star=-(Q_{uu}^{k+1})^{-1}(Q_u^{k+1})-(Q_{uu}^{k+1})^{-1}\circ Q_{ug}^{k+1}(\zeta^{k+1})=0$. By induction, this proves that $\delta u^k_\star=0$ for all $k$, or equivalently that there exists $U_*\in\mathcal{U}$ such that $U\rightarrow U_*$.

It remains to show that $U_*$ is stationary. Since $Q_u^k=0$ for each $k$, \eqref{backv} implies that:
\begin{equation}
\label{ole}
\mathcal{V}_g^k=Q_g^k-Q_{gu}^k\circ(Q_{uu}^k)^{-1}\circ Q_u^k=Q_g^k,\quad\forall k.
\end{equation}
Based on this, we get from \eqref{psi}
\begin{equation*}
\begin{split}
\psi^{H-1} &= \ell_g^{H-1}+(\Phi^{H-1})^*\circ \mathcal{V}_g^{H}\overset{\eqref{Qfunctions}}{=}Q_g^{H-1}\overset{\eqref{ole}}{=}\mathcal{V}_g^{H-1},\\
\psi^{H-2} &= \ell_g^{H-2}+(\Phi^{H-1})^*\circ\psi^{H-1}=\ell_g^{H-2}+(\Phi^{H-1})^*\circ\mathcal{V}_g^{H-1}\\
&\overset{\eqref{Qfunctions}}{=}Q_g^{H-2}\overset{\eqref{ole}}{=}\mathcal{V}_g^{H-2},\\
&\vdots\\
\psi^k&=\mathcal{V}_g^k,\quad\forall k.
\end{split}
\end{equation*}
Hence, \eqref{eq_lemma1} reads: $\dD_{u^k}|_{\bar{U}} J=\ell_u^k+(\text B^k)^{*}\circ\mathcal{V}_g^{k+1}\overset{\eqref{Qfunctions}}{=}Q_u^k=0$.

\bibliographystyle{IEEEtran}
\bibliography{IEEEabrv,mybib_lie}

\end{document}